\documentclass{article}

\usepackage{amsmath}
\usepackage{amsthm}
\usepackage{amssymb}
\usepackage{amscd}
\usepackage[dvipdfmx]{graphicx}
\allowdisplaybreaks[4]

\theoremstyle{plain}
\newtheorem{theo}{Theorem}[section]
\newtheorem{prop}[theo]{Proposition}
\newtheorem{lemm}[theo]{Lemma}
\newtheorem{cor}[theo]{Cororally}

\theoremstyle{definition}
\newtheorem{defi}[theo]{Definition}
\newtheorem{exa}[theo]{Example}

\newtheorem{rmk}[theo]{Remark}

\newcommand{\identitymap}[1]{\mathrm{id}_{#1}} 
\numberwithin{equation}{section}

\newcommand{\vect}[1]{(#1_1,#1_2,#1_3,\ldots)}
\newcommand{\Rep}[1]{Cl_{\mathbb{Z}}(#1)}

\newcommand{\circg}[1]{\langle #1 \rangle}
\newcommand{\Res}[2]{\mathrm{Res}^{#1}_{#2}}

\newcommand{\Lam}[1]{\Lambda(#1)}
\newcommand{\Nr}[1]{Nr(#1)}
\newcommand{\Gh}[1]{#1^{\mathbb{N}}}
\newcommand{\Eva}[2]{\mathrm{E}^{#1}_{#2}}
\newcommand{\supp}{\mathrm{supp}}
\newcommand{\Dpr}{D^{\rm{pr}}}

\newcommand{\Mobius}{M\"{o}bius}
\newcommand{\masm}{\textbf{q}}

\newcommand{\bigzerol}{\smash{\hbox{\large 0}}}       
\newcommand{\bigzerou}{\smash{\lower.3ex\hbox{\large 0}}}  

\begin{document}

\author{Tomoyuki Tamura\\ Graduate school of Mathematics\\ Kyushu University \\ 2017}
\title{Exterior powers of representations of finite groups and integer-valued characters }
\date{}
\maketitle

\begin{abstract}
For given representation of finite groups on a finite dimension complex vector space, we can define  exterior powers of representations. In 1973, Knutson found one of methods of calculating the character of exterior powers of representations with properties of $\lambda$-rings. In this paper \footnote[1]{Keywords and Phrases; Exterior powers representation; finite group; $\lambda$-ring; necklace ring.}, we base this result of Knutson, and relate characters and elements of necklace rings, which were introduced by N.Metropolis and G-C.Rota in 1983, via a generating function of the character of exterior powers of representations. We focus on integer-valued characters and discuss a relation between integer-valued characters and element of necklace rings which has finite support and is contained in some images of truncated operations.
\end{abstract}

%

\section{Introduction}
Let $G$ be an arbitrary group. For a given representation $\rho:G\rightarrow GL(V)$ on a finite dimensional complex vector space $V$, we can define the $i$-th exterior power of the representation $\rho$ by 
\begin{eqnarray*}
&{}&\Lambda^i\rho:G\rightarrow GL(\textstyle\bigwedge^i(V)),\\
&{}&\Lambda^i\rho(g)(v_1\wedge\cdots\wedge v_i):=(\rho(g)v_1)\wedge\cdots\wedge (\rho(g)v_i)
\end{eqnarray*}
for any integer $i\geq 0$, where $g\in G$ and $v_1,\ldots,v_i\in V$.

There exists the problem to find how to calculate exterior powers of representations. It was raised in \cite{Knu}. This problem was raised by Knutson \cite{Knu}. Knutson discovered the method using $\lambda$-rings to calculate the character of $\Lambda^i\rho$ (For $\lambda$-rings, see \S 2.2, \cite{Knu} or \cite{Yau}). This method is as follows: Let $CF(G)$ be the set of all class functions $f:G\rightarrow\mathbb{C}$, which satisfies $f(g^{-1}xg)=f(x)$ for any $x,g\in G$. The set $CF(G)$ is a commutative $\mathbb{C}$-algebra, in particular, is a commutative $\mathbb{Q}$-algebra. Moreover, the set $CF(G)$ has a $\lambda$-ring structure where the $n$-th Adams operation $\psi^n$ satisfies $\psi^n(f)(g):=f(g^n)$ for any $f\in CF(G)$, $g\in G$ and integer $n\geq 1$. Knutson \cite[p.84]{Knu} showed that the character of $\textstyle\Lambda^i\rho$ equals $\lambda^i(\chi)$ for any integer $i\geq 0$ where $\chi$ is the character of a representation $\rho$ and $\lambda^i$ is a $\lambda$-operation of $CF(G)$. 

We note the relation between this result in \cite{Knu} and representation rings and previous studies of representation rings. Let $R(G)$ be a representation ring of $G$, which has a pre-$\lambda$-ring structure whose $\lambda$-operation is defined via exterior powers of representations. The map $X:R(G)\rightarrow CF(G)$ defined by that $X([V])$ is the character of $V$, is an injective ring homomorphism where $[V]$ is an isomorphism class of a representation $V$. From this result in \cite{Knu}, the map $X$ can preserve both of $\lambda$-operations. Thus, the study of a character of the exterior powers of representations is reduced to the study of the $\lambda$-ring structure on representation rings.

We recall previous studies on a representation ring of a group as a $\lambda$-ring after \cite{Knu}. For a set of generators of the representation ring as a $\lambda$-ring, Boorman \cite{Boo} or Bryden \cite{Bry} found the case of symmetric groups or Weyl groups associated with the root systems of type $B_n$ and $D_n$ in 1975 or 1995, respectively. For continuous groups, Osse \cite{Oss} characterized a representation ring of a compact connected Lie group in terms of a $\lambda$-ring. Yau \cite{Yau} arises the problem to characterize a representation ring of a finite group in terms of a $\lambda$-ring.

Because the set of all virtual characters is isomorphic to a representation ring as a $\lambda$-ring, the study of the $\lambda$-ring structure of representation rings is equivalent to the study of the $\lambda$-ring structure of the set of virtual characters. In this paper, we will focus on virtual characters, not representation rings. 

Let $G$ be a finite group. For any character $\chi$ of a representation $\rho: G\rightarrow GL(V)$, Knutson showed that $\lambda^i(\chi)$ is the character of the $i$-th exterior power of $\rho$. We define a power series $\lambda_t(\chi)(g)$ by
\begin{eqnarray}\label{IntroLambda1}
\lambda_t(\chi)(g):=\sum_{i=0}^{\infty}\lambda^i(\chi)(g)t^i
\end{eqnarray}
for any $g\in G$. It has the following form:
\[ \lambda_t(\chi)(g)= \mathrm{det} (I_m+\rho(g)t)= \exp \Big(\sum_{i=1}^{\infty}-\dfrac{\chi(g^i)}{i}(-t)^i \Big) \]
where $I_m$ is the unit matrix whose size is $m=\deg(\rho)$.

\begin{exa}
Let $X$ be a finite $G$-set and let $\chi$ be the permutation character associated with $X$. It is known that $\chi(g)=|\mbox{Fix}(g)|$ holds where 
\[\mbox{Fix}(g):=\{x\in X\ |\ gx=x\} \]
for any $g\in G$. Then we have
\[ \lambda_t(\chi)(g)=\exp\Big(\sum_{i=1}^{\infty}-\dfrac{|\mbox{Fix}(g^i)|}{i}(-t)^i \Big).\]
\end{exa}

We discuss relations between integer-valued characters of finite groups and necklace rings via (\ref{IntroLambda1}). An integer-valued character is a virtual character $\chi$ such that $\chi(g)\in\mathbb{Z}$ holds for any $g\in G$. 

The necklace ring over a commutative ring $R$ was introduced by Metropolis and Rota in 1983 to investigate the notion of Witt rings as a commutative ring. For any commutative ring $R$, the set of all infinite vectors 
\[ a=\vect{a},\quad a_i\in R\]
has a commutative ring structure whose addition ``$+_{Nr}$" is defined by componentwise and multiplication $\alpha=\vect{\alpha}:=x\cdot_{Nr} y$ is defined by
\begin{eqnarray}
\alpha_n:=\sum_{[i, j]=n}(i, j)x_i y_j
\end{eqnarray}
where $x=\vect{x}, y=\vect{y}\in\Gh{R}$. With definitions of the addition ``$+_{Nr}$" and the multiplication ``$\cdot_{Nr}$", we call the set of all infinite vectors the necklace ring over $R$, written by $\Nr{R}$. 

We will relate characters of $G$ and elements of necklace rings as follows: Let $\Lam{R}=1+tR[[t]]$ be a universal $\lambda$-ring with an indeterminate variable $t$. Knutson \cite{Knu} showed that $\Lam{R}$ has a $\lambda$-ring structure where the addition ``$+_{\Lambda}$" is defined by the multiplication of $R[[t]]$. If $R$ is a binomial ring, which is a $\mathbb{Z}$-torsion free $\lambda$-ring whose all Adams operations of $R$ are identity maps, then for each $f\in\Lambda(R)$ there exists a unique element $\vect{a}\in\Nr{R}$ such that 
\begin{eqnarray}\label{IntroLambda2}
f(t)=\prod_{i=1}^{\infty}(1-(-t)^i)^{a_i}.
\end{eqnarray}
holds. Yau \cite[\S 5.6]{Yau} showed that the map $E_{Nr}:\Lam{R}\rightarrow \Nr{R}$ defined by $E_{Nr}(f):=\vect{a}$ is a ring isomorphism. 

For any finite group $G$, the set $CF(G)$, which is the set of all class functions from $G$ to $\mathbb{C}$, also has a binomial ring structure. When $f(t)$ is the generating function $\lambda_t(\chi)\in\Lambda(CF(G))$ in (\ref{IntroLambda2}), it is sufficient to calculate $E_{Nr}(\lambda_t(\chi))$ in order to calculate $\lambda^i(\chi)$. We will relate characters of $G$ and elements of necklace rings using the map $E_{Nr}\circ\lambda_t$.

Now, we outline this paper. In \S 2, we will state basic definitions of necklace rings, $\lambda$-rings, representations of finite groups and the relation among them. 

In \S 3, we will show main theorems of this paper. Let $e$ be the exponent of $G$. Then, a character $\chi$ is an integer-valued character if and only if $a_n=0$ for any integer $n\geq 1$ with $n\nmid e$ where $\vect{a}=E_{Nr}(\lambda_t(\chi))$ (Theorem \ref{maintheorem2}). 

Next, we will discuss the support of an element $\vect{a}\in\Nr{R}$. We will say that $a=\vect{a}$ has the finite support if the number of integers $n\geq 1$ such that $a_n\neq 0$ is finite. By Theorem \ref{maintheorem2}, the element $E_{Nr}(\lambda_t(\chi))$ has the finite support for any integer-valued character $\chi$. We will show that if $E_{Nr}(\lambda_t(\chi))$ has the finite support, then the character $\chi$ is an integer-valued character by Theorem \ref{main4theo3} from an arbitrary $\mathbb{Z}$-torsion free commutative ring. 

Finally, we will discuss the case of a product group of two finite groups $G_1$ and $G_2$. Let $\chi_1$ be a virtual character of $G_1$ and let $\chi_2$ be a virtual character of $G_2$. Then we define a virtual character $\chi_1\chi_2$ of $G_1\times G_2$ by $\chi_1\chi_2((g_1,g_2)):=\chi_1(g_1)\chi_2(g_2) $ for any $g_1\in G_1$ and $g_2\in G_2$. We will show that 
\begin{eqnarray}\label{1main3}
\lambda_t(\chi_1\chi_2)((g_1,g_2))=\lambda_t(\chi_1)(g_1)\cdot_{\Lambda}\lambda_t(\chi_2)(g_2)
\end{eqnarray}
holds in Theorem \ref{main3MMM}, where ``$\cdot_{\Lambda}$" is the multiplication of the universal $\lambda$-ring $\Lambda(\mathbb{C})$.

In addition, we study an equation obtained from (\ref{1main3}) using the map $E_{Nr}$. Under the assumption that characters $\chi_1$ and $\chi_2$ are integer-valued characters, we calculate the multiplication of two elements $E_{Nr}(\lambda_t(\chi_1))$ and $E_{Nr}(\lambda_t(\chi_2))$ in order to calculate the left side of (\ref{1main3}). Note that $E_{Nr}(\lambda_t(\chi_1))$ and $E_{Nr}(\lambda_t(\chi_2))$ have the finite support. Moreover, we will give another form of the multiplication of two elements of necklace rings which have finite support with Frobenius operations of necklace rings.

In \S 4, we consider the character $\chi$ of a representation of $S_n$ which will be defined with an $R$-matrix, and the calculation method of $E_{Nr}(\lambda_t(\chi))$ as an example
Here, we denote the following notations in this paper.
\begin{itemize}
\item[(i)] Let $\mathbb{N}$, $\mathbb{Z}$ or $\mathbb{Q}$ be the set of all positive integers, integers or rational integers, respectively. 
\item [(ii)]For any integers $i$ and $j$, a symbol $i\mid j$ stands for that $i$ divides $j$ and $i\nmid j$ stands for that $i$ does not divide $j$.
\item [(iii)]For any integers $i$ and $j$, a symbol $(i,j)$ stands for the greatest common divisor of $i$ and $j$, and a symbol $[i,j]$ stands for the least common multiple of $i$ and $j$.
\item [(iv)]We denote the\ \Mobius\ function by $\mu$.
\item [(v)]For any set $X$, we denote the identity map on $X$ by $\identitymap{X}$ defined by $\identitymap{X}(x)=x$ for any $x\in X$.
\item [(vi)]Let $\mathbb{C}$ be the complex field, and let $M(k,\mathbb{C})$ be the set of all complex matrices whose size is $k\times k$. 
\item [(vii)] We denote the unit element of a finite group by $1$.
\item [(viii)] For each element $g$ of a finite group $G$, we denote by $\circg{g}$ the subgroup of $G$ generated by $g\in G$, and by $O(g)$ the order of $g$. We define the exponent of $G$ by the least common multiple of all $O(g)$'s for $g\in G$.
\item [(ix)] We assume that every ring and semiring have the unit element, written by $1$.
\end{itemize}

\section*{Acknowledgement}
I would like to thank my supervisor, Professor Hiroyuki Ochiai, for his appropriate advice even under such circumstances. Without this advice, I believe that this paper was not completed.

\section{Preliminaries}
In this section, we will define some notations to state main results of this paper in \S 3. In \S 2.1, we will discuss necklace rings and operations. In \S 2.2, we will state definitions and properties of pre-$\lambda$-rings and $\lambda$-rings. In \S 2.3, we will discuss a relation between a character of exterior powers of representations of finite groups and $\lambda$-rings.

\subsection{Necklace rings and operations}\label{PMPrelim}
Main references of \S \ref{PMPrelim} are \cite{MR} and \cite{VW}. Given a commutative ring $R$, let $R^{\mathbb{N}}$ be the set of all infinite vectors whose elements belong to $R$. The set $\Gh{R}$ has a commutative ring structure whose addition and multiplication are defined by componentwise. The zero element of $\Gh{R}$ is $(0,0,0,\ldots)$ and the unit element of $\Gh{R}$ is $(1,1,1,\dots,)$. In addition, we define another commutative ring structure on $R^{\mathbb{N}}$. Its addition ``$+_{Nr}$" is defined by componetwise, and multiplication ``$\cdot_{Nr}"$ by
\begin{eqnarray*}
\alpha_n:=\sum_{[i,j]=n}(i,j)x_iy_j
\end{eqnarray*}
where $x=\vect{x},y=\vect{y}\in \Gh{R}$ and $\alpha=\vect{\alpha}=x\cdot_{Nr} y$. If we regard $R^{\mathbb{N}}$ as a commutative ring with such operations, then we call $R^{\mathbb{N}}$ the necklace ring over $R$, written by $\Nr{R}$. The zero element of $\Nr{R}$ is $(0,0,0,\ldots)$ and the unit element of $\Nr{R}$ is $(1,0,0,\ldots)$.

Next, we define a map $\phi:\Nr{R}\rightarrow \Gh{R}$ by $\phi(x):=\vect{a}$ where
\[ a_n:=\sum_{d\mid n}dx_d,\quad x=\vect{x}\]
for any $x \in\Nr{R}$ and integer $n\geq 1$. 

\begin{prop}\label{prophi}
The map $\phi$ is a ring homomorphism. If $R$ is $\mathbb{Z}$ torsion-free, then the map $\phi$ is injective. If $R$ is a $\mathbb{Q}$-algebra, then the map $\phi$ is bijective.
\end{prop}
\begin{proof}
The fact that $\phi$ is a ring homomorphism was proved in \cite[Lemma 5.4]{Yau}.

We assume that $R$ is $\mathbb{Z}$-torsion free, and prove that the map $\phi$ is injective. Let $x=\vect{x}$ be an element of $\Nr{R}$ and assume $\vect{a}=\phi(x)=0$. We prove $x_n=0$ by induction on $n\geq 1$. By the definition of the map $\phi$, we have $x_1=a_1=0$. Let $n\geq 2$ be an integer and we assume that $x_k=0$ holds for any integer $k$ with $k<n$. We use the definition of the map $\phi$ again, so we have
\[ na_n=-\sum_{d\mid n, d\neq n}dx_d=0. \]
Since $R$ is $\mathbb{Z}$-torsion free, we have $x_n=0$. Hence, we have $x=0$.

Next, we assume that $R$ is a $\mathbb{Q}$-algebra and prove that the map $\phi$ is bijective. It is enough to prove that the map $\phi$ is surjective, because a $\mathbb{Q}$-algebra commutative ring is $\mathbb{Z}$-torsion free. For any $a=\vect{a}\in\Gh{R}$, put $x=\vect{x}\in\Nr{R}$ satisfying that 
\[ x_n=\dfrac{1}{n}\sum_{d\mid n}\mu\Big(\dfrac{n}{d}\Big)a_d\]
for any integer $n\geq 1$. We have $\phi(x)=a$ by the\ \Mobius\ inversion formula, that is, the map $\phi$ is surjective.
\end{proof}
For any integer $r\geq 1$, we define maps
\begin{eqnarray*}
V_r&:&\Nr{R}\rightarrow \Nr{R},\\
V_r&:&\Gh{R}\rightarrow \Gh{R},\\
F_r&:&\Nr{R}\rightarrow \Nr{R},\\
F_r&:&\Gh{R}\rightarrow \Gh{R}
\end{eqnarray*}
as follows: Let $x=\vect{x}\in\Nr{R}$. Then $\vect{y}=V_r(x)$ and\\ $\vect{z}=F_r(x)$ satisfy
\begin{eqnarray*}
y_n&:=&\begin{cases} x_{n/r} & \mbox{if}\ r\mid n, \\ 0 & \mbox{if}\ r\nmid n,\end{cases}\\
z_n&:=& \sum_{[j,r]=nr}\dfrac{j}{n}x_j\\
\end{eqnarray*}
for any integer $n\geq 1$. 

Let $a=\vect{a}\in\Gh{R}$. Then $\vect{b}=V_r(a)$ and \\$\vect{c}=F_r(a)$ satisfy
\begin{eqnarray*}
b_n&:=&\begin{cases} ra_{n/r} & \mbox{if}\ r\mid n, \\ 0 & \mbox{if}\ r \nmid n,\end{cases}\\
c_n&:=&a_{nr}
\end{eqnarray*}
for any integer $n\geq 1$. 

In this paper, we call $V_r$ the $r$-th Verschiebung operation, and $F_r$ the $r$-th Frobenius operation, respectively \cite{MR} and \cite{VW}.
\begin{prop}\label{proVF}
The operations $V_r$ and $F_s$ satisfy the following relations.
\begin{itemize}
\item[{\rm (1)}] $V_r\circ V_s=V_{rs},$
\item[{\rm (2)}] $F_r\circ F_s=F_{rs},$
\item[{\rm (3)}] $F_r\circ V_r=r\identitymap{R},$
\item[{\rm (4)}] $F_r\circ V_s=V_s\circ F_r$ {\rm if} $(r,s)=1$,
\item[{\rm (5)}] $F_r\circ V_s=(r,s)F_{r/(r,s)}\circ V_{s/(r,s)}$.
\end{itemize}
\end{prop}
\begin{proof}
For necklace ring $\Nr{R}$, identities (1), (2) and (3) were proved in \cite[Theorem 1]{MR}, (4) was proved in \cite[Theorem 5.4]{VW} and (5) was proved in \cite[Proposition 5.5]{VW}. So, we consider this proposition in the case of $\Gh{R}$. Let $a=\vect{a}$ be an element of $\Gh{R}$ and let $n\geq 1$ be an integer.
\begin{itemize}
\item [(1)] Let $b=\vect{b}=V_r(a)$ and $c=\vect{c}=V_s(b)$. If $rs\mid n$, then $c_n=b_{n/s}=a_{n/rs}$ holds. If $rs\nmid n$, then we have $c_n=0$. Hence, we have $V_r\circ V_s=V_{rs}$.
\item [(2)] Let $b=\vect{b}=F_r(a)$ and $c=\vect{c}=F_s(b)$. Then, we have $c_n=b_{sn}=a_{rsn}$, that is, we have $F_r\circ F_s=F_{rs}$.
\item [(3)] Let $b=\vect{b}=V_r(a)$ and $c=\vect{c}=F_r(b)$. Then, we have $c_n=b_{rn}=ra_n$, that is, we have $F_r\circ V_r=r\identitymap{R}$.
\item [(4)] Let $b=\vect{b}=V_s(a)$, $c=\vect{c}=F_r(b)$,\\
$e=\vect{e}=F_r(a)$ and $f=\vect{f}=V_s(e)$. We prove $c_n=f_n$. If $n \mid s$, then we have $c_n=b_{nr}=sa_{nr/s}=sd_{n/s}=f_n$. If $n\nmid s$, we have $c_n=f_n=0$. In any case, we have $c_n=f_n$ for any integer $n\geq 1$, that is, we have $F_r\circ V_s=V_s\circ F_r$.
\item [(5)] We have $F_r\circ V_s=F_{r/(r,s)}\circ F_{(r,s)}\circ V_{(r,s)}\circ V_{s/(r,s)}=(r,s)F_{r/(r,s)}\circ V_{s/(r,s)}$ by identities (1), (2) and (3).
\end{itemize}
\end{proof}
\begin{prop}\label{main4lemm0}
The operation $F_r$ is a ring homomorphism for any integer $r\geq 1$.
\end{prop}
\begin{proof}
First, we prove this proposition in the case of $\Nr{R}$. Obviously, the operation $F_r$ is an additive homomorphism. Thus, we consider the multiplication. Let $x=\vect{x}$ and $y=\vect{y}$ be elements of $\Nr{R}$. Put $\alpha=\vect{\alpha}=x\cdot_{Nr}y$, $u=\vect{u}=F_r(x)$, $v=\vect{v}=F_r(y)$ and $w=\vect{w}=F_r(\alpha)$. Then, have
\begin{eqnarray*}
w_n&=&\sum_{[j,r]=nr}\dfrac{j}{n}\alpha_j\\
&=&\sum_{[p,q]=j, [j,r]=nr}\dfrac{j}{n}(p,q)x_p y_q\\
&=&\sum_{[p,q,r]=nr}\dfrac{1}{n}px_p qy_q\\
&=&\sum_{[i,j]=n}\dfrac{1}{n}\Big(\sum_{[p,r]=ir}px_p\Big)\Big(\sum_{[q,r]=jr}qy_q\Big)\\
&=&\sum_{[i,j]=n}(i,j)\Big(\sum_{[p,r]=ir}\dfrac{p}{i}x_p\Big)\Big(\sum_{[q,r]=jr}\dfrac{q}{j}y_q\Big)\\
&=&\sum_{[i,j]=n}(i,j)u_iv_j.
\end{eqnarray*}
In addition, we have $F_r(1)=1$ by the definition of $F_r$ because if an integer $n\geq 1$ satisfies $[1,r]=nr$, then $n=1$ holds. Hence, the operation $F_r$ is a ring homomorphism in the case of $\Nr{R}$.

Next, we consider the case of $\Gh{R}$. Let $a=\vect{a}$ and $b=\vect{b}$ be elements of $\Gh{R}$. Put $e=\vect{e}=F_r(a)$ and $f=\vect{f}=F_r(b)$. Thus, we have $e_nf_n=a_{nr}b_{nr}$, which is the $n$-th element of $F_r(ab)$. In addition, we have $F_r(1)=1$, and hence, the operation $F_r$ is a ring homomorphism.
\end{proof}

\begin{prop}\label{phiVF}
The map $\phi:\Nr{R}\rightarrow \Gh{R}$ preserves operations $V_r$ and $F_r$.
\end{prop}
\begin{proof}
First, we prove $V_r\circ\phi=\phi\circ V_r$ for any integer $r\geq 1$. Let $x=\vect{x}$ be an element of $\Nr{R}$, and let $n\geq 1$ be an integer. We put $y=\vect{y}=V_r(x)$, $a=\vect{a}=\phi(y)$, $b=\vect{b}=\phi(x)$ and $c=\vect{c}=V_r(b)$. If $n\nmid r$, then all divisors of $n$ do not divide $r$. Thus we have $a_n=0$ and $c_n=\sum_{d\mid n}d b_d=0$. If $n\mid r$, then $c_n=\sum_{d\mid n/r}dr x_d=\sum_{d\mid n}dy_d=a_n$ holds. In any case we have $x_n=a_n$ for any integer $n\geq 1$, that is, we have $V_r\circ\phi=\phi\circ V_r$.

Next, we prove $F_r\circ\phi=\phi\circ F_r$. For any integer $n\geq 1$, we have 
\[f_n=\sum_{d\mid n}d\Big(\sum_{[j,r]=dr}\dfrac{j}{d}x_j\Big)=\sum_{[j,r]\mid nr}jx_j=e_n\]
where $e=\vect{e}=F_r\circ\phi(x)$ and $f=\vect{f}=\phi\circ F_r(x)$. Note that $j$ satisfies $[j,r]\mid nr$ if and only if $j\mid nr$, and $e_n$ is the $nr$-th element of $\phi(x)$. Hence, we have $F_r\circ\phi=\phi\circ F_r$.
\end{proof}

In the remainder of \S \ref{PMPrelim}, we suppose that a commutative ring $R$ is a $\mathbb{Q}$-algebra. For any integer $r\geq 1$, we define two maps 
\begin{eqnarray*}
V_r'&:&\Nr{R}\rightarrow\Nr{R},\\
V_r'&:&\Gh{R}\rightarrow \Gh{R}
\end{eqnarray*}
by 
\[ V_r':= \dfrac{1}{r}V_r.\]
The operation $V_r'$ is said to be the $r$-th divided Verschiebung operation, which was defined as operations of aperiodic rings in \cite{VW}. This definition comes from the definition in aperiodic rings.

\begin{prop}\label{prodV1}
The map $\phi$ preserves divided Verschiebung operations.
\end{prop}
\begin{proof}
By Proposition \ref{phiVF}, we have $\phi\circ V'_r=\phi\circ\dfrac{1}{r}V_r=\dfrac{1}{r}V_r\circ\phi=V'_r\circ\phi$. Hence, the map $\phi$ preserves divided Verschiebung operations.
\end{proof}

\begin{prop}
For any integer $r\geq 1$, the operation $V'_r$ is a homomorphism of the additive and the multiplication.
\end{prop}
\begin{proof}
In both of cases, it is obvious that the operation $V'_r$ is an additive homomorphism. Hence, we consider the multiplication.

First, we consider the case of $\Gh{R}$. Let $a=\vect{a}$ and $b=\vect{b}$ be elements of $\Gh{R}$. Put $c=\vect{c}=ab$, \\
$e=\vect{e}=V'_r(a)$, $f=\vect{f}=V'_r(b)$ and $g=\vect{g}=V'_r(c)$. If $r\mid n$, then $g_n=c_{n/r}=a_{n/r}b_{n/r}=e_{n}f_n$ holds. If $r\nmid n$, then $g_n=e_n f_n=0$ holds. Hence, the operation $V'_r$ satisfies $V'_r(ab)=V'_r(a)V'_r(b)$.

In case where $\Nr{R}$, we use the map $\phi$ which is bijective by Proposition \ref{prophi}. The operation $V'_r$ is equal to $\phi^{-1}\circ V'_r \circ \phi$ by Proposition \ref{prodV1}. Hence, the operation $V'_r$ satisfies $V'_r(x\cdot_{Nr} y)=V'_r(x)\cdot_{Nr}V'_r(y)$ for any $x,y\in\Nr{R}$.
\end{proof}

The followings also hold by Proposition \ref{proVF}.
\begin{prop}
The operations $V'_r$ and $F_r$ satisfy the following relations.
\begin{itemize}
\item [{\rm (1)}]$V_r'\circ V_s'=V_{rs}'$,
\item [{\rm (2)}]$F_r\circ V_r'=\identitymap{R}$,
\item [{\rm (3)}]$F_r\circ V_s'=V_s'\circ F_r$ {\rm if} $(r,s)=1$,
\item [{\rm (4)}]$F_r\circ V_s'=F_{r/(r,s)}\circ V_{s/(r,s)}'$.
\end{itemize}
\end{prop}

\subsection{Pre-$\lambda$-rings and $\lambda$-rings}\label{PreLambda}
We discuss pre-$\lambda$-rings and $\lambda$-rings. Main references of \S \ref{PreLambda} are \cite[Chapter 13]{Hus}, \cite{Knu} and \cite{Yau}.

A $\lambda$-semiring is a commutative semiring $R$ with operations $\lambda^n:R\rightarrow R$, $n=0,1,2,\ldots$, such that
\begin{itemize}
\item [{\rm (1)}]$\lambda^0(r)=1$,
\item [{\rm (2)}]$\lambda^1(r)=r $,
\item [{\rm (3)}]$\lambda^n(r+s)=\sum_{i+j=n}\lambda^i(r)\lambda^j(s)$
\end{itemize}
for any $r,s\in R$. A semi-$\lambda$-ring $R$ is said to be pre-$\lambda$-ring if $R$ has a commutative ring structure with the addition and the multiplication of $R$ as a semiring.

We define a $\lambda$-semiring homomorphism $f:R\rightarrow R'$ between two $\lambda$-semirings $R$ and $R'$ by a semiring homomorphism satisfying $f\circ\lambda^i=\lambda^i\circ f$ for any integer $i\geq 0$.

\begin{prop}[\text {\cite [p.172]{Hus}}]\label{RingComp}
Let $(R^*,\theta)$ be the ring completion of the semiring $R$. There exists a pre-$\lambda$-ring structure on $R^*$ uniquely such that the map $\theta$ is a $\lambda$-semiring homomorphism.
\end{prop}

Let $R$ be a commutative ring. Next, we define $\Lambda(R)$ by
\[ \Lambda(R)=\left\{f=1+\sum_{i=1}^{\infty}r_it^i\in R[[t]] \mid r_i\in R\right\} \]
where $t$ is an indeterminate variable. 
In \cite{Knu} and \cite{Yau}, the set $\Lambda(R)$ is called universal $\lambda$-ring, and has a pre-$\lambda$-ring structure whose addition is defined by the multiplication of $R[[t]]$. For the multiplication and $\lambda$-operations of $\Lambda(R)$, see \cite{Knu} and \cite{Yau}.

%
%
%
%
Let $R$ be a pre-$\lambda$-ring. For any $r\in R$, we define $\lambda_t(r)\in\Lambda(R)$ by
\[ \lambda_t(r)=\sum_{i=0}^{\infty}\lambda^i(r)t^i.\]
By the definition of pre-$\lambda$-rings, we have $\lambda_t(r+s)=\lambda_t(r)\lambda_t(s)$ for any $r,s\in R$, which shows that the map $\lambda_t:R\rightarrow\Lambda(R)$ is an additive homomorphism.

We call a $\lambda$-semiring homomorphism between two pre-$\lambda$-rings a \\ \textrm{pre-$\lambda$-homomorphism}.

A pre-$\lambda$-ring $R$ is said to be a $\lambda$-ring if the map $\lambda_t:R\rightarrow\Lambda(R)$ is a pre-$\lambda$-homomorphism. Knutson \cite{Knu} showed that $\Lambda(R)$ is a $\lambda$-ring for any commutative ring $R$. A pre-$\lambda$-homomorphism between two $\lambda$-rings is called $\lambda$-homomorphism. 

A $\lambda$-subring of a $\lambda$-ring $R$ is defined by a subring $R'\subset R$ such that $\lambda^i(x)\in R'$ for all integers $i\geq 0$ and $x\in R'$.

For a $\lambda$-homomorphism and a $\lambda$-subring, the following proposition holds.

\begin{prop}[\text {\cite[Proposition 1.27 (2)]{Yau}}]\label{subring}
Let $R$ and $S$ be $\lambda$-rings and let $f:R\rightarrow S$ be a $\lambda$-homomorphism. The image of $f$ is a $\lambda$-subring of $S$.
\end{prop}

Let $R$ be a $\lambda$-ring. For each integer $n\geq 1$, we define the $n$-th Adams operation $\psi^n:R\rightarrow R$ by the following equation:
\[ \sum_{n=1}^{\infty}\psi^n(r)(-t)^n:=-t\dfrac{d}{dt}\log\lambda_t(r)\quad (\text{for any}\ r\in R).\]
For Adams operations, the following propositions and theorem hold.
\begin{prop}\label{Adams1}
Let $R$ and $S$ be $\lambda$-rings and let $f:R\rightarrow S$ be a ring homomorphism. If $f$ is a $\lambda$-homomorphism, then $f\circ\psi^n=\psi^n\circ f$ holds for any integer $n\geq 1$. 
\end{prop}
\begin{proof}
Assume that $f$ is a $\lambda$-homomorphism. By the definition of Adams operations, for any integer $n\geq 1$ and $r\in R$, an element $\psi^n(r)$ is a polynomial of $\lambda^i(r)$, $i=0,1,\ldots, n$. Hence $f\circ \psi^n=\psi^n\circ f$ holds.
\end{proof}

\begin{prop}[\text {\cite[Corollary 3.16]{Yau}}]
Let $R$ and $S$ be $\lambda$-rings and let $f:R\rightarrow S$ be a ring homomorphism. The map $f$ is a $\lambda$-homomorphism if $f\circ\psi^n=\psi^n\circ f$ holds for any integer $ n\geq 1$ and $S$ is $\mathbb{Z}$-torsion free.
\end{prop}

\begin{theo}[\text {\cite[Theorem 3.6]{Yau}}]\label{Adams2}
On any $\lambda$-ring $R$ and integer $n\geq 1$, the $n$-th Adams operation $\psi^n$ is a $\lambda$-homomorphism. 
\end{theo}

Let $R$ be a commutative ring. We define a map $z:\Lambda(R)\rightarrow \Gh{R}$ by $z(f):=\vect{a}$ where 
\[ \sum_{i=1}^{\infty} a_i(-t)^i=-t\dfrac{d}{dt}\log(f(t))\quad (\text{for any}\ f\in\Lambda(R)).\]
The map $z$ is a ring homomorphism. If $R$ is $\mathbb{Z}$-torsion free, then the map $z$ is injective (\cite[p.133]{Yau}).

A commutative ring $R$ equipped with ring homomorphisms $\psi^n$ for $n\geq 1$, such that $\psi^1=\identitymap{R}$ and $\psi^n\circ\psi^m=\psi^{nm}$ hold for any integers $n,m\geq 1$, is said to be a $\psi$-ring. 

For example, the set $\Gh{R}$ is a $\psi$-ring with Frobenius operations. All $\lambda$-rings are $\psi$-rings with Adams operations. Conversely, the following theorem holds.
\begin{theo}[\text {\cite[p.50]{Knu}}]\label{LambdaGh}
Suppose that $R$ is a $\mathbb{Q}$-algebra and a $\psi$-ring with operations $\psi^n, n\geq 1$. Then, there exists the $\lambda$-ring structure on $R$ uniquely such that $\psi^n$ is an $n$-th Adams operation.
\end{theo}

Next, we state the definition of binomial rings. It is a $\mathbb{Z}$-torsion free $\lambda$-ring such that all Adams operations are identity maps. 
There exists the following relation between the universal $\lambda$-ring $\Lambda(R)$ and the necklace ring $\Nr{R}$.
\begin{lemm}{\rm\cite[Lemma 5.41]{Yau}}\label{ENRdefine}
Suppose that $R$ is a binomial ring. For each $f\in\Lambda(R)$, we can write uniquely as
\[ f(t)=\prod_{i=1}^{\infty}(1-(-t)^i)^{a_i} \]
for some elements $a_1,a_2,a_3,\ldots\in R$.
\end{lemm}
By Lemma \ref{ENRdefine}, we can state the following theorem.
\begin{theo}{\rm\cite[Theorem 5.42]{Yau}}\label{YauTheo}
Suppose that $R$ is a binomial ring. Then the map $E_{Nr}:\Lambda(R)\rightarrow \Nr{R}$ defined by 
\[ E_{Nr}\Big(\prod_{i=1}^{\infty}(1-(-t)^i)^{a_i}\Big)=\vect{a}, \]
is a ring isomorphism, and $z=\phi\circ E_{Nr}$ holds.
\end{theo}

We define three types of ring-homomorphisms. Let $R$ and $S$ be commutative rings and $F:R\rightarrow S$ be a ring homomorphism. We define maps $\Lam{F}:\Lam{R}\rightarrow \Lam{S}$, $\Nr{F}:\Nr{R}\rightarrow \Nr{S}$ and $\Gh{F}:\Gh{R}\rightarrow \Gh{S}$ by
\begin{eqnarray*}
\Lam{F}\Big(1+\sum_{i=1}^{\infty}a_it^i\Big)&:=&1+\sum_{i=1}^{\infty}F(a_i)t^i, \\
\Nr{F}(\vect{x})&:=&(F(x_1),F(x_2),F(x_3),\ldots) \\
\Gh{F}(\vect{y})&:=&(F(y_1),F(y_2),F(y_3),\ldots)
\end{eqnarray*}
for any $a_1,a_2,\ldots\in R$, $\vect{x}\in\Nr{R}$ and $\vect{y}\in\Gh{R}$. 

\begin{prop}\label{triplecomm}
Let $F:R\rightarrow S$ be a ring homomorphism between two rings $R$ and $S$.
\begin{itemize}
\item [{\rm (1)}] Three maps $\Lambda(F), \Nr{F}$ and $\Gh{F}$ are ring homomorphisms.
\item [{\rm (2)}] One has $z\circ\Lambda(F)=\Gh{F}\circ z$ and $\phi\circ \Nr{F}=\Gh{F}\circ \phi$. If $R$ and $S$ are binomial rings, then $\Nr{F}\circ E_{Nr}=E_{Nr}\circ\Lambda(F)$ holds.
\end{itemize}
\end{prop}
\begin{proof}
(1) For $\Lambda(F)$, see \cite[p.104]{NO}, and for $\Nr{F}$, see \cite[p.18]{VW}. For $\Gh{F}$,
\begin{eqnarray*}
&{}&\Gh{F}(\vect{a})+\Gh{F}(\vect{b})\\
&=&(F(a_1),F(a_2),F(a_3),\ldots)+(F(b_1),F(b_2),F(b_3),\ldots)\\
&=&(F(a_1+b_1),F(a_2+b_2),F(a_3+b_3),\ldots)\\
&=&\Gh{F}(a_1+b_1,a_2+b_2,a_3+b_3,\ldots),\\
&{}&\Gh{F}(\vect{a})\Gh{F}(\vect{b})\\
&=&(F(a_1),F(a_2),F(a_3),\ldots)(F(b_1),F(b_2),F(b_3),\ldots)\\
&=&(F(a_1b_1),F(a_2b_2),F(a_3b_3),\ldots)=\Gh{F}(a_1b_1,a_2b_2,a_3b_3,\ldots),\\
&{}&\Gh{F}(1,1,1,\ldots)=(1,1,1,\ldots)
\end{eqnarray*}
holds for any $\vect{a}, \vect{b}\in\Gh{R}$. Then, the map $\Gh{F}$ is a ring homomorphism.

(2) We prove $z\circ\Lambda(F)=\Gh{F}\circ z$. For any integer $n\geq 1$, let $Z_n\vect{x}$ be a polynomial which satisfies that the $n$-th element of $z(1+\sum_{i=1}^{\infty}a_it^i)$ is $Z_n\vect{a}$ where $a_1,a_2,\ldots\in R$. Then, the $n$-th element of $z\circ\Lambda(F)(f)$ and $\Gh{F}\circ z(f)$ coincide with $Z_n(F(a_1), F(a_2), F(a_3),\ldots)$ for any $f=1+\sum_{i=1}^{\infty}a_it^i$. Then, we have $z\circ\Lambda(F)=\Gh{F}\circ z$.

Similarly, we show that both $n$-th elements of $\phi\circ \Nr{F}(x)$ and $\Gh{F}\circ \phi(x)$ are equal in order to prove the second equation of $(2)$, where $x=\vect{x}\in\Nr{R}$. Clearly, it is $\sum_{d\mid n}dF(x_d)$. Then, we have $\phi\circ \Nr{F}=\Gh{F}\circ \phi$.

To prove the third equation of $(2)$, we use Lemma \ref{ENRdefine}. One has 
\begin{eqnarray*}
\phi\circ\Nr{F}\circ E_{Nr}=\Gh{F}\circ\phi\circ E_{Nr}=\Gh{F}\circ z\\
=z\circ\Lambda(F)=\phi\circ E_{Nr}\circ\Lambda(F)=\phi\circ E_{Nr}\circ\Lambda(F).
\end{eqnarray*}
By the assumption, the ring $S$ is $\mathbb{Z}$-torsion free. Hence $\Nr{F}\circ E_{Nr}=E_{Nr}\circ\Lambda(F)$ holds by Proposition \ref{prophi}.
\end{proof}

By the definition of $\Lam{F}$, $\Nr{F}$ and $\Gh{R}$, we have the following proposition.
\begin{prop}
For any ring homomorphisms $F:R\rightarrow S$ and $G:S\rightarrow T$, we have $\Lam{G\circ F}=\Lam{G}\circ\Lam{F}$, $\Nr{G\circ F}=\Nr{G}\circ\Nr{F}$ and $\Gh{(G\circ F)}=\Gh{G}\circ\Gh{F}$.
\end{prop}

\subsection{Exterior powers of representations of finite groups}
In \S 2.3, we discuss representations of finite groups and, in particular, a character of exterior powers of representations of finite groups using $\lambda$-rings. Main references of \S 2.3 is \cite{Knu}.

The set of all representations of a finite group $G$ has an equivalent relation which is defined by isomorphisms of representations of $G$. Let $M(G)$ be the set of equivalence classes. For any representation $V$, we denote an equivalences class of $V$ by $[V]\in M(G)$. The set $M(G)$ has a $\lambda$-semiring structure whose addition is defined via the direct sum of two representations, multiplication is defined via the tensor product of two representations, and $\lambda$-operations are defined via exterior powers of representations.

We define the representation ring of $G$, written by $R(G)$, by the ring completion of the $\lambda$-semiring $M(G)$. By Proposition \ref{RingComp}, the representation ring $R(G)$ has a pre-$\lambda$-ring structure, and the representation ring $R(G)$ is a free abelian group with equivalence classes of irreducible representations of $G$ as generators (see Appendix A.6).

Let $CF(G)$ be the set of all class functions. It has a $\mathbb{C}$-algebra $\psi$-ring structure with the following operations:
\begin{eqnarray*}
\begin{cases}
(f_1+f_2)(g):=f_1(g)+f_2(g),\\
(f_1f_2)(g):=f_1(g)f_2(g), \\
(cf)(g):=cf(g),\\
\psi^n(f)(g):=f(g^n)
\end{cases} 
\end{eqnarray*}
for any $f, f_1, f_2\in CF(G), c\in\mathbb{C}$, integer $n\geq 1$ and $g\in G$. In particular, the set $CF(G)$ is a $\mathbb{Q}$-algebra. Thus, by Theorem \ref{LambdaGh} and the above operations, the set $CF(G)$ has a $\lambda$-ring structure.

For any representation $\rho:G\rightarrow GL(V)$, we denote the character of $V$ by $X(V)$, which satisfies that $X(V)(g)=\mathrm{Tr}(\rho(g))$ for any $g\in G$. The character of a representation satisfies the followings for any representations $V$ and $W$.
\begin{itemize}
\item [{\rm (1)}]$X(V)=X(W)$ whenever two representations $V$ and $W$ are isomorphic,
\item [{\rm (2)}]$X(V\oplus W)=X(V)+X(W)$,
\item [{\rm (3)}]$X(V\otimes W)=X(V)X(W)$.
\end{itemize}
Then, the mapping $V\mapsto X(V)$ defines the ring homomorphism $X:R(G)\rightarrow CF(G)$ by that $X([V])=X(V)$ holds for any representation $V$.

\begin{theo}{\rm\cite[p.84]{Knu}}\label{mainly}
The map $X$ is an injective pre-$\lambda$-homomorphism. In particular, the representation ring $R(G)$ is a $\lambda$-ring.
\end{theo}

By Theorem \ref{mainly}, a character of exterior powers of representations can be written with $\lambda$-operations of $CF(G)$. For more detail, for any representation $V$ whose character is $\chi$, the character of the $i$-th exterior power of $\rho$ is $\lambda^i(\chi)$.

We define $\Rep{G}$ by the image of the map $X$. It is a free abelian group with irreducible characters of $G$ as generators, and is a $\lambda$-ring. We call an element of $\Rep{G}$ a virtual character.

For any $g\in G$, we define the map $\Eva{G}{g}:CF(G)\rightarrow\mathbb{C}$ by $\Eva{G}{g}(f):=f(g)$ for any $f\in CF(G)$. It is a ring homomorphism. By this define we can write $\lambda_t(\chi)(g)$ which is (\ref{IntroLambda1}) by 
\[ \lambda_t(\chi)(g)=\Lambda(\Eva{G}{g})(\lambda_t(\chi))=\sum_{i=0}^{\infty}\lambda^i(\chi)(g)t^i\]
for any character $\chi$ of $G$ and $g\in G$.

\subsection{Algebraic integers}
In \S2.4, we consider algebraic integers for \S 3.3. For more detail, see \cite[\S 17]{CR}.

We define a monic polynomial $f(t)\in\mathbb{Z}[t]$ by $a_n=1$ where $f(t)=\sum_{i=0}^{n}a_it^i$. In this paper, we define an algebraic integer $\alpha$ by that there exists an irreducible monic polynomial $f(t)\in\mathbb{Z}[t]$ such that $f(\alpha)=0$ holds. 

\begin{theo}[\text{\cite[(17.2)]{CR}}]\label{alge1}
Any zero of a monic polynomial in $\mathbb{Z}[t]$ is an algebraic integer.
\end{theo}

By Theorem \ref{alge1}, we see that a primitive $e$-th roof of unity, written by $\omega$ is also an algebraic integer with $f(t)=t^e-1$.

To state the main theorems, we consider the following lemmas.
\begin{lemm}[\text {\cite[(17.7)]{CR}}]\label{alge2}
If $\alpha$ and $\beta$ are algebraic integers, then $\alpha+\beta$ and $\alpha\beta$ are also algebraic integers.
\end{lemm}

\begin{lemm}[\text {\cite[p.106, Exercise]{CR}}]\label{alge3}
A rational number is an algebraic integer if and only if it is an integer.
\end{lemm}
\begin{proof}
Let $\alpha$ is a rational number. If $\alpha$ is an algebraic integers, then there exists an irreducible polynomial $f\in\mathbb{Z}[t]$ such that $f(\alpha)=0$ holds. Hence, we have $t-\alpha \mid f(t)$. Moreover, a polynomial $f(t)$ is an irreducible as a polynomial of $\mathbb{Q}[t]$ by the Gauss Lemma (see Lemma \ref{Gauss}). Then, we have $f(t)=t-\alpha$. Hence, we have $\alpha\in\mathbb{Z}$.

Conversely, if $\alpha\in\mathbb{Z}$ then we put $f(t)=t-\alpha$. Then, an integer $\alpha$ is an algebraic integer.
\end{proof}

\section{Main Result}\label{Prelim}
In this section, we focus on integer-valued characters of finite groups as main results of this paper with \S 2. First, we recall the definition of integer-valued characters.
\begin{defi}
A virtual character $\chi$ will be said to be an integer-valued character if $\chi(g)$ is an integer for any $g\in G$.
\end{defi}

In \S\ref{SSExample}, we introduce two examples of integer-valued characters. In \S\ref{SSTruncated}, we define truncated operations, on $\Nr{R}$ and $\Gh{R}$. In \S\ref{SSInteger}, we characterize integer-valued characters in terms of $\lambda$-rings and necklace rings as a main result. In \S\ref{SSFinitesupport}, we introduce the notion of support of elements of $\Nr{R}$ and cycle of $\Gh{R}$ for any commutative ring $R$, and state the relation between integer-valued characters and elements of necklace rings which have finite support. In \S\ref{SSMultiFrobenius}, we discuss product groups and Frobenius operations to calculate $\lambda_t(\chi)(g)$ for a virtual character $\chi$ and an element $g$ of a finite group. Moreover, we discuss the multiplication of two elements of the image of truncation operations with Frobenius operations.

\subsection{Examples of integer-valued characters}\label{SSExample}
Now, we state two examples of integer-valued characters.
\begin{exa}
Permutation characters are integer-valued characters. See Appendix A.4.
\end{exa}
Next, we consider characters of symmetric groups.

\begin{lemm}\label{SN1}
Put $\sigma=(1\ 2\cdots n)\in S_n$ and let $d\geq 1$ be an integer.
\begin{itemize}
\item [{\rm (1)}] If $d$ and $n$ are coprime, then $\sigma^d$ is also a cycle of length $n$. 
\item [{\rm (2)}] If $d\mid n$, then we have
\[ \sigma^d=(1\ d+1\ 2d+1\cdots(d'-1)d+1)\cdots (d\ d+d\cdots (d'-1)d+d) \]
where $d'=n/d$.
\end{itemize}
\end{lemm}
\begin{proof}
For any integer $i=1,\ldots,n$, put $a_{k,i}=\sigma^k(i)$. Then, we have $ki=b_{k,i}n+a_{k,i}$ for some integer $b_{k,i}\geq 0$. 

First, we assume that $k$ and $n$ are coprime. If $a_{k,i}=a_{k,j}$, then $k(i-j)=(b_{k,i}-b_{k,j})n$ holds. Thus, the integer $n$ divides $k(i-j)$, that is, the integer $n$ divides $i-j$. Hence, we have $i=j$, that is, $\sigma^k=(a_{k,1} \cdots a_{k,n})$ holds. 

Next, we assume $d\mid n$. If $i\leq (d'-1)d$, then $\sigma^d(i)=d+1$ holds. If $i=(d'-1)d+k$ where $k=1,\ldots, d'$, then $\sigma^d(i)=k$ holds. Hence, we have $\sigma^d=(1\ d+1\ 2d+1\cdots(d'-1)d+1)\cdots (d\ d+d\cdots (d'-1)d+d)$.
\end{proof}

\begin{exa}
A finite group is called a Q-group if all characters are integer-valued \cite[p.8]{Kle}. A typical example of Q-group is a symmetric group, and more generally, Weyl groups.

Here we recall the proof \cite[p.10]{Kle} of the fact that a symmetric group is a Q-group.

Let $\sigma$ and $\tau$ be elements of $S_n$ satisfying $\circg{\sigma}=\circg{\tau}$. Then, there exists an integer $k\geq 1$ such that $\sigma^k=\tau$ holds and $k$ is coprime to $O(\sigma)$ by Lemma \ref{APP1}. Let $\{\lambda_1,\ldots,\lambda_m\}$ be the cycle structure of $\sigma$. Then, the integer $k$ is coprime to $\lambda_i$ for any $i=1,\ldots,m$ by Lemma \ref{APP3}. Thus, the element $\tau$ has the same cycle structure by Lemma \ref{SN1}, that is, elements $\sigma$ and $\tau$ are conjugate by Lemma \ref{APP4}. Hence, a character of $S_n$ is an integer-valued character by \cite[Proposition 9]{Kle}.
\end{exa}


\subsection{Truncated operations}\label{SSTruncated}

Let $R$ be a commutative ring. In \S \ref{SSTruncated}, we define operations on $\Nr{R}$ and $\Gh{R}$ to state main theorems in \S\ref{SSInteger}, \S\ref{SSFinitesupport} and \S\ref{SSMultiFrobenius}, and discuss relations among operations.

\begin{defi}
For any integer $r\geq 1$, we define 
\begin{eqnarray*}
T_r&:&\Nr{R}\rightarrow \Nr{R},\\
T_r&:&\Gh{R}\rightarrow \Gh{R}
\end{eqnarray*}
by $y=\vect{y}=T_r(x)$ and $b=\vect{b}=T_r(a)$ where $x=\vect{x}\in\Nr{R}$, $a=\vect{a}\in\Gh{R}$, and
\begin{eqnarray*}
y_n&:=&\begin{cases} x_n & \mbox{if}\ n\mid r, \\ 0 & \mbox{if}\ n\nmid r, \end{cases} \\
b_n&:=&a_{(n,r)}.
\end{eqnarray*}
We call $T_r$ the $r$-th truncated operation.
\end{defi} 

In the following proposition, we consider the image of $T_r$.
\begin{prop}\label{ImageTr}
Let $r\geq 1$ be an integer.
\begin{itemize}
\item [{\rm (1)}] An element $x=\vect{x}\in\Nr{R}$ belongs to the image of $T_r$ if and only if $x_n=0$ holds for any integer $n\geq 1$ with $n\nmid r$.
\item [{\rm (2)}] An element $a=\vect{a}\in\Gh{R}$ belongs to the image of $T_r$ if and only if $a_n=a_{(n,r)}$ holds for any integer $n\geq 1$.
\end{itemize}
\end{prop}
\begin{proof}
We prove the statement (1). Let $x=\vect{x}$ be an element of the image of $T_r$. Then, there exists $y\in\Nr{R}$ such that $x=T_r(y)$ holds. For any integer $n\geq 1$ with $n\nmid r$, the element $x_n$ is the $n$-th element of $T_r(y)$, which is equal to $0\in R$. 

Next, we assume that $x_n=0$ holds for any integer $n\geq 1$ with $n\nmid r$. Put $u=\vect{u}=T_r(x)$. 
If an integer $n\geq 1$ satisfies $n\mid r$, then $x_n=u_n$ holds. If $n\nmid r$, then $x_n=u_n=0$ holds. In any case, we have $x=u=T_r(x)$, that is, the element $x$ belongs to the image of $T_r$.

(2) Let $a=\vect{a}$ be the element of the image of $T_r$. Then, there exists $b=\vect{b}\in\Nr{R}$ such that $a=T_r(b)$ holds. Hence we have $a_n=b_n=b_{(n,r)}=a_{(n,r)}$ for any integer $n\geq 1$.

Next, we assume that $a_n=a_{(n,r)}$ holds for any integer $n\geq 1$. Put $c=\vect{c}=T_r(a)$. Thus, we have $c_n=a_{(n,r)}=a_n$ for any integer $n\geq 1$, and hence, we have $a=c=T_r(a)$, that is, the element $a$ belongs to the image of $T_r$.
\end{proof}

\begin{prop}\label{main4lemm00}
The operation $T_r$ is a ring homomorphism for any integer $r\geq 1$.
\end{prop}
\begin{proof}
Obviously, the operation $T_r$ is an additive homomorphism and $T_r(1)=1$ holds in both cases. Thus, we consider the multiplication.

First, we prove this proposition in the case of $\Nr{R}$. Let \\$x=\vect{x}$ and $y=\vect{y}$ be elements of $\Nr{R}$. We put $\alpha=\vect{\alpha}=x\cdot_{Nr} y$, $u=\vect{u}=T_r(x)$, $v=\vect{v}=T_r(y)$ and $w=\vect{w}=T_r(\alpha)$. For any integers $r,n\geq 1$, the $n$-th elements of $T_r(x)\cdot_{Nr}T_r(y)$ is 
\begin{eqnarray}\label{Tr1}
\sum_{[i,j]=n}(i,j)u_iv_j
\end{eqnarray}
by the definition of multiplication of $\Nr{R}$. If $n\mid r$, then integers $i$ and $j$ satisfying $[i,j]=n$ divide $r$. Thus, (\ref{Tr1}) coincides with
\[ \sum_{[i,j]=n}(i,j)x_iy_j=\alpha_n=w_n. \]
If $n\nmid r$, then (\ref{Tr1}) is $0\in R$ because either $i$ or $j$ with $[i,j]=n$ does not divide $r$ and $w_n=0$ holds. Hence, we have $T_r(x)\cdot_{Nr}T_r(y)=w=T_r(\alpha)=T_r(x\cdots_{Nr}y)$.

Next, we prove this proposition in the case of $\Gh{R}$. Let $a=\vect{a}$ and $b=\vect{b}$ be elements of $\Gh{R}$. We put $c=\vect{c}=ab$, $e=\vect{e}=T_r(a)$, $f=\vect{f}=T_r(b)$ and $g=\vect{g}=T_r(c)$. For any integer $r,n\geq 1$, we have $e_nf_n=a_{(n,r)}b_{(n,r)}=c_{(n,r)}=g_n$, that is, the operation $T_r$ is a ring homomorphism.
\end{proof}

\begin{prop}\label{main4lemm01}
For any integers $r,s\geq 1$, we have $T_r\circ T_s=T_{(r,s)}$.
\end{prop}
\begin{proof}
First, we prove this proposition in the case of $\Nr{R}$. Let \\ $x=\vect{x}$ be an element of $\Nr{R}$, and put $y=\vect{y}=T_s(x)$ and $\alpha=\vect{\alpha}=T_r(y)$. For any integer $d\geq 1$, if $d\mid (r,s)$, which implies $d\mid r$ and $d\mid s$, then we have $\alpha_d=y_d=x_d$. If $d\mid r$ and $d\nmid s$, then we have $\alpha_d=y_d=0$. If $d\nmid r$, then we have $\alpha_d=0$. Hence, we have $T_r\circ T_s=T_{(r,s)}$.

Next, we prove this proposition in the case of $\Gh{R}$. Let $x=\vect{x}$ be an element of $\in\Nr{R}$, and put $y=\vect{y}=T_s(x)$ and $\alpha=\vect{\alpha}=T_r(y)$. Then, we have $\alpha_n=y_{(n,r)}=x_{(n,r,s)}$ for any integer $n\geq 1$. Hence $T_r\circ T_s=T_{(r,s)}$ holds.
\end{proof}

\begin{prop}\label{lemm13}
The map $\phi:\Nr{R}\rightarrow \Gh{R}$ preserves truncated operations.
\end{prop}
\begin{proof}
For any integer $r\geq 1$, we prove $T_r\circ\phi=\phi\circ T_r$. Let $x=\vect{x}$ be an element of $\Nr{R}$. We put $a=\vect{a}=\phi(x)$, $b=\vect{b}=T_r(a)$, $y=\vect{y}=T_r(x)$ and $c=\vect{c}=\phi(y)$. Then, we have
\begin{eqnarray*}
c_n=\sum_{d\mid n}dy_d=\sum_{d\mid n, d\mid r}dx_d=\sum_{d\mid (n,r)}dx_d=a_{(n,r)}=b_n.
\end{eqnarray*}
Hence $\phi\circ T_r=T_r\circ\phi$ holds.
\end{proof}
By Proposition \ref{lemm13}, we can state the following corollary.
\begin{cor}\label{lemm131}
Suppose that $R$ is $\mathbb{Z}$-torsion free. Let $x=\vect{x}$ be an element of $\Nr{R}$, let $r\geq 1$ be an integer, and put $a=\vect{a}=\phi(x)$. Then, the followings are equivalent.
\begin{itemize}
\item [{\rm (1)}] For any integer $n\geq 1$ with $n\nmid r$, we have $x_n=0$.
\item [{\rm (2)}] For any integer $n\geq 1$, we have $a_n=a_{(n,r)}$.
\end{itemize}
\end{cor}
\begin{proof}
First, we assume the statement (1). We have $x=T_r(x)$ by definition of truncated operations. Thus, we have $a=\phi(x)=\phi\circ T_r(x)=T_r\circ\phi(x)=T_r(a)$. The $n$-th element of $a$ and $T_r(a)$ are $a_n$ and $a_{(n,r)}$, that is, the statement (2) holds by Proposition \ref{ImageTr} (2).

Conversely, we assume the statement (2). Thus, we have $\phi(x)=a=T_r(a)=T_r\circ\phi(x)=\phi\circ T_r(x)$. Since $R$ is $\mathbb{Z}$-torsion free, we have $x=T_r(x)$. For any integer $n\geq 1$ with $n\nmid r$, the $n$-th element of $T_r(x)$ is $0\in R$, and hence, $x_n=0$ holds by Proposition \ref{ImageTr}(1).
\end{proof}

Next, we state relations among truncated operations, Frobenius operations and Verschiebung operations.

\begin{prop}\label{ComCF}
For any integers $r,s\geq 1$, we have $F_r\circ T_s=T_{s/(r,s)}\circ F_{(r,s)}$.
\end{prop}
Note that Proposition \ref{ComCF} holds on both $\Nr{R}$ and $\Gh{R}$. To prove Proposition \ref{ComCF}, we introduce the following notation.

\begin{defi}
For any $u\in R$ and integer $n\geq 1$, we define an element $u\delta_n$ of $\Nr{R}$ by
\begin{eqnarray*}
u\delta_n=\vect{x},\quad x_n:=\begin{cases} u & \mbox{if}\ m=n, \\ 0 & \mbox{otherwise}.\end{cases}
\end{eqnarray*}
\end{defi}

For elements $u\delta_n$, the following lemma holds.

\begin{lemm}\label{ComCF11}
For any $u\in R$ and integers $n,r \geq 1$, we have $F_r(u\delta_n)=(n,r)u\delta_{n/(n,r)}$.
\end{lemm}
\begin{proof}
By the definition of Frobenius operations, we have
\begin{eqnarray}\label{ComCF1}
y_k=\sum_{[j,r]=kr}\dfrac{j}{k}x_j
\end{eqnarray}
where $x=\vect{x}=u\delta_n$ and $y=\vect{y}=F_r(x)$. If $[n,r]=kr$ does not hold, then the right side of (\ref{ComCF1}) is equal to $0\in R$. If $[n,r]=kr$ holds, which implies $k=n/(n,r)$, then the right side of (\ref{ComCF1}) is equal to $(n,r)u$. Hence, this lemma holds.
\end{proof}

\begin{lemm}\label{allTr}
For any $x,y\in\Nr{R}$, we have $x=y$ if $T_r(x)=T_r(y)$ holds for any integer $r\geq 1$.
\end{lemm}
\begin{proof}
Put $x=\vect{x}$ and $y=\vect{y}$. For any integer $n\geq 1$, both of $n$-th element of $T_n(x)$ and $T_n(y)$ are $x_n$ and $y_n$, and these are equal. Hence, we have $x=y$.
\end{proof}

\begin{proof}[Proof of Proposition \ref{ComCF}]
First, we prove $F_r\circ T_s(x)=T_{s/(r,s)}\circ F_{(r,s)}(x)$ when $x=u\delta_n$ for $u\in R$ and integer $n\geq 1$. 

If $n\mid s$, then $F_r\circ T_s(u\delta_n)=(n,r)u\delta_{n/(n,r)}$ and 
\[ T_{s/(r,s)}\circ F_{(r,s)}(u\delta_n)=(n,r)T_{s/(r,s)}(u\delta_{n/(n,r)})=(n,r)u\delta_{n/(n,r)}\]
hold by Lemma \ref{ComCF11}.

If $n\nmid s$, then $F_r\circ T_s(u\delta_n)=0$ holds by the definition of truncated operations, and we have
\[T_{s/(r,s)}\circ F_{(r,s)}(u\delta_n)=(n,r)T_{s/(r,s)}(u\delta_{n/(n,r)})=0\]
by Lemma \ref{ComCF11}. Thus, we can show that this proposition holds when $x=u\delta_n$.

Next, we consider an arbitrary element $x=\vect{x}\in\Nr{R}$. For any integer $n\geq 1$, we have
\begin{eqnarray*}
T_n(F_r\circ T_s(x))&=&\sum_{d\mid n}(F_r\circ T_s)x_d\delta_d\\
&=&\sum_{d\mid n}(T_{s/(r,s)}\circ F_{(r,s)})x_d\delta_d \\
&=&T_n(T_{s/(r,s)}\circ F_{(r,s)}(x)).
\end{eqnarray*}
Thus, we have $T_n(F_r\circ T_s(x))=T_n(T_{s/(r,s)}\circ F_{(r,s)}(x))$ for any integer $n\geq 1$. By Lemma \ref{allTr}, we have $F_r\circ T_s(x)=T_{s/(r,s)}\circ F_{(r,s)}(x)$, and hence, this proposition holds in the case of $\Nr{R}$.

Finally, we prove this proposition in the case of $\Gh{R}$. For any integers $r,s,n\geq 1$, we have
\[ (s, (r,s)n)=(s,rn,sn)=(s,rn).\]
Thus, for any $a=\vect{a}\in\Gh{R}$ and integer $n\geq 1$, we have
\[b_n=a_{(s,rn)}=a_{(s,(r,s)n)}=a_{(s,rn)}=c_n\]
where $b=\vect{b}=F_r\circ T_s(a)$ and $c=\vect{c}=T_{s/(r,s)}\circ F_{(r,s)}(a)$.
Hence, this proposition holds in the case of $\Gh{R}$.
\end{proof}

\begin{prop}\label{proTV}
For any integers $r,s \geq 1$ we have 
\[ T_r\circ V_s=\begin{cases}V_s\circ T_{r/s} & \mbox{if}\ s\mid r, \\ 0 & \mbox{if}\ s\nmid r. \end{cases} \]
\end{prop}
\begin{proof}
First, we consider the case of $\Nr{R}$. Let $x=\vect{x}$ be an element of $\Nr{R}$ and let $n\geq 1$ be an integer. Put $y=\vect{y}=T_r\circ V_s(x)$.

We assume that $s\mid r$ and put $\alpha=\vect{\alpha}=V_s\circ T_{r/s}(x)$. If $s\mid n$ and $n\mid r$, then $y_n=x_{n/s}$ and $\alpha_n=x_{n/s}$ holds from $n/s\mid r/s$. If $s\nmid n$, then $y_n=0$ and $\alpha_n=0$ hold. If $n\nmid r$, then $y_n=0$ holds and $\alpha_n=0$ holds since $n/s \nmid r/s$ holds. Thus, we have $y=\alpha$, that is, one has $T_r\circ V_s=V_s\circ T_{r/s}$. 

Assume that $s\nmid r$. If $n\nmid r$, then $y_n=0$ holds. If $n\mid s$ then $y_n=0$ holds and the $(r,s)$-element of $V_r(x)$ is $0\in R$ from $r\nmid s$. Thus, this proposition holds for $\Nr{R}$.

Next, we consider the case of $\Gh{R}$. Let $a=\vect{a}$ be an element of $\Gh{R}$. Put $b=\vect{b}=V_s(a)$ and $c=\vect{c}=T_r(b)$. We assume $s\mid r$, then $s\mid (n,r)$ holds if and only if $s\mid n$. Thus, we have
\[c_n=b_{(n,r)}=a_{(n,r)/s}=a_{(n/s,r/s)}=e_n\]
where $e=\vect{e}=V_s\circ T_{r/s}(a)$. If $s\nmid n$ then $c_n=0$ holds by $s\nmid (n,r)$, and $e_n=0$ holds. If $s\nmid r$, we have $c_n=0$ by $s\nmid (n,r)$ for any $n\geq 1$.

As a result, this proposition holds for both cases.
\end{proof}

By Proposition \ref{proTV}, we have the following corollary.
\begin{cor}\label{proTV2}
Suppose that $R$ is a $\mathbb{Q}$-algebra. For any integers $r,s \geq 1$ we have 
\[ T_r\circ V'_s=\begin{cases}V'_s\circ T_{r/s} & \mbox{if}\ s\mid r, \\ 0 & \mbox{if}\ s\nmid r.\end{cases}\]
\end{cor}

\subsection{Integer-valued characters}\label{SSInteger}

In \S \ref{SSInteger}, we discuss relations among integer-valued characters of a finite group $G$, $\lambda$-operations of $CF(G)$ and images of truncated operations.

First, we show the following result.
\begin{theo}\label{maintheorem1}
If a virtual character $\chi$ is an integer-valued character, then $\lambda^i(\chi)$ is also an integer-valued character for any integer $i\geq 0$.
\end{theo}
To prove Theorem \ref{maintheorem1}, we characterize integer-valued characters in terms of a $\lambda$-ring. In the remainder of this section, we denote by $e$ the exponent of $G$ (see (viii) of the preface of this paper).

\begin{lemm}\label{lemm11}
For any virtual character $\chi$ and $g\in G$, there exists a polynomial $f\in\mathbb{Z}[t]$ such that $\psi^n(\chi)(g)=f(\omega^n)$ holds for any integer $n\geq 1$ where $\omega$ is a primitive $e$-th root of unity. In particular, an element $\chi(g)$ is an algebraic integer.
\end{lemm}
\begin{proof}
We may assume that $\chi$ is the character of a representation $\rho: G\rightarrow GL(V)$ by the definition of $R(G)$. All eigenvalues of the linear map $\rho(g): V\rightarrow V$ are $e$-th roots of unity, so the set of eigenvalues of $\rho(g)$ is written as $\omega^{a_1},\ldots,\omega^{a_m}$ with some non-negative integers $a_1,\ldots,a_m$.

We put $f(t) = t^{a_1}+\cdots+t^{a_m} \in \mathbb{Z}[t]$. Then 
\[ \psi^n(\chi)(g) = \chi(g^n) = \mathrm{Tr}(\rho(g)^n)= \omega^{na_1}+\cdots+\omega^{na_m}= f(\omega^n)\]
holds. 

The last statements of this lemma holds using Lemmas \ref{alge2} and \ref{alge3} because
 elements $\omega^{a_1},\dots,\omega^{a_m}$ are algebraic integers.
\end{proof}


By Lemma \ref{lemm11}, we have the following Lemma \ref{lemm12}. 

\begin{lemm}\label{lemm12}
A virtual character $\chi$ is an integer-valued character if and only if $\psi^n(\chi)=\psi^{(n,e)}(\chi)$ holds for any integer $n\geq 1$.
\end{lemm}
\begin{proof}
In this proof, we denote a primitive $e$-th root of unity by $\omega$, and we denote the Galois group of the field extension $\mathbb{Q}(\omega^{(n,e)}):\mathbb{Q}$ by $\Gamma_{n}$ for any integer $n\geq 1$. For more detail of Galois groups, see \cite{Ste}.

First, we assume that the character $\chi$ is an integer-valued character. Let $g$ be an element of $G$. By Lemma \ref{lemm11}, there exists a polynomial $f\in\mathbb{Z}[t]$ such that $\psi^k(\chi)(g)=f(\omega^k)$ holds for any integer $k\geq 1$. By the assumption, $\psi^n(\chi)(g)$ and $\psi^{(n,e)}(\chi)(g)$ are integers. 

For any integer $n\geq 1$, let $n'$ be $n/(n,e)$. Then we have $\psi^n(\chi)(g)=f(\omega^{(n,e)n'})$ and $\psi^{(n,e)}(\chi)(g)=f(\omega^{(n,e)})$, which belong to $\mathbb{Q}(\omega^{(n,e)})$. Thus, since $(n', e/(n,e))=1$ holds there exists an element $\sigma\in\Gamma_{n}$ such that $\sigma(\omega^{(n,e)})=\omega^{(n,e)n'}$ holds. Hence we have 
\begin{eqnarray}\label{lemm121}
\psi^n(\chi)(g)=\sigma(\psi^{(n,e)}(\chi)(g))=\psi^{(n,e)}(\chi)(g).
\end{eqnarray}

Conversely, we assume that $\psi^n(\chi)=\psi^{(n,e)}(\chi)$ holds for any integer $n\geq 1$. In particular, $\psi^n(\chi)=\chi $ holds for any integer $n\geq 1$ with $(n,e)=1$. Use Lemma \ref{lemm11}, there exists a polynomial $f\in\mathbb{Z}[t]$ such that $\psi^k(\chi)(g)=f(\omega^k)$ holds for any integer $k\geq 1$. 

Let $\sigma$ be an element of $\Gamma_{e}$. Then, an integer $n$ satisfying $\sigma(\omega)=\omega^n$ is coprime to $e$. Thus, we have
\[ \sigma(\chi(g))=\sigma(f(\omega))=f(\omega^n)=\psi^n(\chi)(g)=\chi(g).\]
Then, $\chi(g)$ is a rational number, which implies that $\chi(g)$ is an integer by Lemma \ref{alge2} and Lemma \ref{lemm11}.
\end{proof}

\begin{proof}[Proof of Theorem \ref{maintheorem1}]
The element $\chi$ is a virtual character. By Proposition \ref{subring} and Theorem \ref{mainly}, elements $\lambda^i(\chi)$ are also virtual characters for all integers $i\geq 0$.
Since $CF(G)$ is a $\lambda$-ring, the $n$-th Adams operation $\psi^n$ is a $\lambda$-homomorphism by Theorem \ref{Adams2}. Let $\chi$ be an integer-valued character. By Lemma \ref{lemm12}, we have
\[ \psi^n(\lambda^i(\chi))=\lambda^i(\psi^n(\chi))=\lambda^i(\psi^{(n,e)}(\chi))=\psi^{(n,e)}(\lambda^i(\chi))\]
for any integers $i\geq 0$ and $n\geq 1$. We use Lemma \ref{lemm12} again, as a result, $\lambda^i(\chi)$ is an integer-valued character.
\end{proof}

Next, we characterize integer-valued characters with necklace rings and truncated operations.

\begin{theo}\label{maintheorem2}
A virtual character $\chi$ is an integer-valued character if and only if $E_{Nr}(\lambda_t(\chi))$ belongs to the image of the truncated operation $T_e$. In particular, if $\chi$ satisfies such conditions, we have the following finite product form of $\lambda_t(\chi)$,
\[ \lambda_t(\chi)=\prod_{d\mid e}(1-(-t)^d)^{\alpha_d}\]
where $\alpha=\vect{\alpha}=E_{Nr}(\lambda_t(\chi))$.
\end{theo}

\begin{proof}
By Lemma \ref{lemm12}, a virtual character $\chi$ is an integer-valued character if and only if $\psi^n(\chi)=\psi^{(n,e)}(\chi)$ holds for any integer $n\geq 1$. The map $\phi:\Nr{CF(G)}\rightarrow \Gh{CF(G)}$ preserves the truncated operation $T_e$ by Proposition \ref{lemm13}, and the $n$-th element of $\phi(\alpha)=\phi(E_{Nr}(\lambda_t(\chi))=z(\lambda_t(\chi))$ is $\psi^n(\chi)$ for any $n\geq 1$. Thus, $\psi^n(\chi)=\psi^{(n,e)}(\chi)$ holds for any integer $n\geq 1$ if and only if $\alpha_n=0$ holds for any integer $n\geq 1$ with $n\mid e$ by Corollary \ref{lemm131}. Hence, the element $\alpha$ belongs to the image of the operation $T_e$ by Proposition \ref{ImageTr}.
\end{proof}

Let $\chi$ be a virtual character and $\alpha=\vect{\alpha}=E_{Nr}(\lambda_t(\chi))$. In Theorem \ref{maintheorem2} we saw that $\alpha_n=0$ holds for any integer $n\geq 1$ with $n\nmid e$ if $\chi$ is an integer-valued character. 

Next, we fix an element $g\in G$ satisfying that $\chi(g^k)$ is an integer for any integer $k\geq 1$, and study a power series $\lambda_t(\chi)(g)$. Put $\vect{\beta}=\Nr{\Eva{G}{g}}(\alpha)$, and we show that the number of integers $n\geq 1$ such that $\beta_n\neq 0$ is equal to, or less than the number of divisors $d$ of $e$ by the following theorem.

\begin{theo}\label{maintheorem3}
Let $\chi$ be a virtual character and we denote $\alpha=E_{Nr}(\lambda_t(\chi))$. For any $g\in G$ satisfying that $\chi(g^k)$ is an integer for any integer $k\geq 1$, all elements of $\Nr{\Eva{G}{g}}(\alpha)$ are integers and $\Nr{\Eva{G}{g}}(\alpha)$ belongs to the image of $T_{O(g)}$. In particular, the power series $\lambda_t(\chi)(g)$ has the following form,
\[ \lambda_t(\chi)(g)=\prod_{d\mid O(g)}(1-(-t)^d)^{\beta_d}\]
where $\vect{\beta}=\Nr{\Eva{G}{g}}(\alpha)$.
\end{theo}

To prove Theorem \ref{maintheorem3}, we show the following lemma.
\begin{lemm}\label{lemm3-1}
For any subgroups $H$ of $G$, the restriction map $\mathrm{Res}^G_H: CF(G)\rightarrow CF(H)$ is a $\lambda$-homomorphism.
\end{lemm}
\begin{proof}
By Proposition \ref{Adams1} and that $CF(H)$ is $\mathbb{Z}$-torsion free, we show $\psi^n\circ\Res{G}{H}=\Res{G}{H}\circ\psi^n$ for any integer $n\geq 1$. For any $\chi\in CF(G)$ and $h\in H$, we have
\[\psi^n(\Res{G}{H}(\chi))(h)=\chi(h^n)=\Res{G}{H}(\psi^n(\chi))(h). \]
Hence, this lemma holds.
\end{proof}

\begin{proof}[Proof of Theorem \ref{maintheorem3}]
By the assumption, $\Res{G}{\circg{g}}(\chi)$ is an integer-valued character of $\circg{g}$. Thus, $\lambda_t(\chi)(g)$ belongs to $\Lambda(\mathbb{Z})$ by Lemma \ref{lemm3-1} and Theorem \ref{maintheorem1}. The set of all integers $\mathbb{Z}$ is a binomial ring with the binomial coefficient (\cite{Yau} \S5.1). Hence, all elements of $\Nr{\Eva{G}{g}}(\alpha)$ are integers. 

Next, we consider $\Nr{\Eva{G}{g}}(\alpha)$. One has
\[ \Nr{\Eva{G}{g}}(\alpha)=\Nr{\Eva{G}{g}}\circ \Nr{\Res{G}{\circg{g}}}(\alpha)=\Nr{\Eva{G}{g}}\circ E_{Nr}(\lambda_t({\Res{G}{\circg{g}}}(\chi))).\]
The exponent of a cyclic subgroup $\circg{g}$ is $O(g)$. Hence $\Nr{\Eva{G}{g}}(\alpha)$ belongs to the image of $T_{O(g)}$.
\end{proof}

\subsection{Finite support elements of a necklace ring}\label{SSFinitesupport}
Let $R$ be a commutative ring. First, we define the following sets and properties associate with elements of $\Nr{R}$ and $\Gh{R}$.

\begin{defi}
For each $x=\vect{x}\in\Nr{R}$, we define $\supp(x)$ by the set of integers $n\geq 1$ such that $x_n\neq 0$ holds. If $\supp(x)$ is finite, then we call that $x$ has the finite support. 
\end{defi}

\begin{lemm}\label{supplemm}
An element of the image of $T_r$ has the finite support for any integer $r\geq 1$.
\end{lemm}
\begin{proof}
Let $x$ be an element of the image of $T_r$. Then, the set $\supp(x)$ is contained the set of all divisor of $r$. In particular, the set $\supp(x)$ is finite.
\end{proof}

\begin{defi}
Let $c\geq 1$ be an integer. An element $a=\vect{a}\in\Gh{R}$ has the $c$-cycle if $a_i=a_j$ holds for any integers $i\geq j \geq 1$ with $c\mid i-j$.
\end{defi}
All elements of the image of the truncated operation $T_r$ has the $r$-cycle. As another example, all elements of $z\circ\lambda_t(CF(G))$ have the $e$-cycle by the definition of Adams operation of $CF(G)$ (see \S 2.3), where $e$ is the exponent of a finite group $G$.

Let $G$ be a finite group, let $\chi$ be a virtual character of $G$, and put $\alpha=\vect{\alpha}=E_{Nr}(\lambda_t(\chi))$. In \S \ref{SSInteger}, we prove that if a virtual character $\chi$ is an integer-valued character, then $\alpha_n=0$ holds for any integer $n\geq 1$ satisfying that $n$ does not divide the exponent of $G$. In particular, the element $E_{Nr}(\lambda_t(\chi))$ has the finite support.

In \S\ref{SSFinitesupport}, we show that if $E_{Nr}(\lambda_t(\chi))$ has the finite support, then $\chi$ is an integer-valued character by the following theorem.
\begin{theo}\label{main4theo3}
Suppose that $R$ is $\mathbb{Z}$-torsion free. For any element $x\in\Nr{R}$ satisfying that $\phi(x)$ has the $c$-cycle, the element $x$ has the finite support if and only if $x$ belongs to the image of $T_c$. 
\end{theo}
By Theorem \ref{main4theo3} as $R=CF(G)$ and Theorem \ref{maintheorem2}, we obtain the following corollary.
\begin{cor}\label{main4theo33}
A virtual character $\chi$ is an integer-valued character if and only if $E_{Nr}(\lambda_t(\chi))$ has the finite support.
\end{cor}
 
First, we state the following lemma.

\begin{lemm}\label{main4lemm3}
If an element of $x \in\Nr{R}$ has the finite support, then $F_r(x)$ also has the finite support for any integer $r\geq 1$.
\end{lemm}
\begin{proof}
Put $y=F_r(x)$. We prove that if $m \in \supp(y)$, then there exists integer $j$ such that $j \geq m$ and $j \in \supp(x)$ hold.

By the definition of Frobenius operation $F_r$, we have
\[ y_m=\sum_{[j,r]=mr}\dfrac{j}{m}x_j.\]
Since $y_m \neq 0$ holds, there exists an integer $j$ such that $[j,r]=mr$ and $j \in \supp(x)$ hold. Hence, we have $mr =[j,r]\leq jr$, and then $m \leq j$.
\end{proof}

\begin{lemm}\label{main4lemm2}
If an element $a\in\Gh{R}$ has the $c$-cycle, then $F_r(a)$ has the $c/(c,r)$-cycle for any integer $r\geq 1$.
\end{lemm}
\begin{proof}
Put $b=\vect{b}=F_r(a)$. For any integers $k_1, k_2\geq1$, we have
\begin{eqnarray*}
b_{(k_1c/(c,r))+k_2}=a_{(k_1rc/(c,r))+k_2r}=a_{k_2r}=b_{k_2}.
\end{eqnarray*}
Hence, the element $b$ has the $c/(c,r)$-cycle.
\end{proof}

\begin{lemm}\label{main4lemm4}
Let $r\geq 1$ be an integer with $(c,r)=1$. If an element $a\in \Gh{R}$ has the $c$-cycle, then there exists an integer $s\geq 1$ such that $r\mid s $ and $a=F_{s}(a)$ holds.
\end{lemm}
\begin{proof}
By $(c,r)=1$, there exist integers $p<0$ and $q>0$ such that $pc+br=1$ holds. 

We put $s=qr$. Then $r\mid s$. If $a=\vect{a}$ and $b=\vect{b}=F_s(a)$, then
\[b_n = a_{sn} = a_{(1-pc)n} = a_n\]
holds for any integer $n\geq 1$, that is, $a=F_s(a)$ holds.
\end{proof}

We define the following maps.
\begin{defi}
For each $n\geq 1$, we define a map $d_n:\mathbb{N}\rightarrow\mathbb{N}$ as follows: The number $d_n(k)$ is the maximum number satisfying that $d_n(k)\mid k$ holds and $d_n(k)$ and $n$ are coprime, for any integer $k\geq 1$.
\end{defi}

\begin{lemm}\label{main4lemm5}
Suppose that $R$ is $\mathbb{Z}$-torsion free, and let $x$ be an element of $\Nr{R}$ satisfying that $\phi(x)$ has the $c$-cycle. If the set $d_c(\supp(x))$ is finite, then the set $d_c(\supp(x))$ has the unique element $1\in\mathbb{N}$.
\end{lemm}
\begin{proof}
Let $n$ be an element of $\supp(x)$. We show $d_c(n)=1$. By the assumption, we can put the least common multiple of all elements of $d_c(\supp(x))$, written by $l$. Two integers $l$ and $c$ are coprime. Thus, there exists an integer $s$ such that $l\mid s$ and $\phi(x)=F_s(\phi(x))$ holds by Lemma \ref{main4lemm4}. The map $\phi$ preserves the Frobenius operation $F_s$, thus we have $\phi(x)=F_s(\phi(x))=\phi(F_s(x))$. Since $R$ is $\mathbb{Z}$-torsion free, we have $x=F_s(x)$. Thus,
\begin{eqnarray}
 x_n=\sum_{[j,s]=ns}\dfrac{j}{n}x_j
\end{eqnarray}
holds where $x=\vect{x}$. 

By $x_n\neq 0$, there exists an integer $j\in\supp(x)$ such that $[j,s]=ns$ holds. In particular, the integer $j$ satisfies $j/(j,s)=n$. By the definition of the map $d_c$, we have $d_c(j)\mid j$. On the other hand we have $d_c(j)\mid s$ by $j\in\supp(x)$. Thus we have $d_c(j)\mid (j,s)$. Hence $d_c(n)=d_c(j/(j,s))=1$ holds.
\end{proof}

\begin{proof}[Proof of Theorem \ref{main4theo3}]
We assume that $x$ belongs to the image of $T_c$. Then, the element $x$ has the finite support by Lemma \ref{supplemm}. Conversely, we assume that $x$ has the finite support, and we show $x_n=0$ for any integer $n\geq 1$ with $n\nmid c$ where $x=\vect{x}$. 

We denote the prime decomposition of $c$ and $n$ by $c=p_1^{e_1}\cdots p_r^{e_r}$ and $n=p_1^{f_1}\cdots p_r^{f_r}\ (e_i,f_i\geq 0, i=1,\ldots,r$). Let $i=1,\ldots,r$ be an integer, and put $c_i=c/p_i^{e_i}$. Since the set $\supp(x)$ is finite by Lemma \ref{main4lemm3}, the set $d_{c_i}(\supp(F_{p_i^{e_i}}(x)))$ is also finite. By Lemma \ref{main4lemm2}, the element $\phi(F_{p_i^{e_i}}(x))$ has the $c_i$-cycle. Using Lemma \ref{main4lemm5}, we have 
\begin{eqnarray}\label{main4theo31}
d_{c_i}(\supp(F_{p_i^{e_i}}(x)))=\{1\}.
\end{eqnarray}
By $n\nmid c$, there exists an integer $k=1,\ldots,r$ such that $p_k\mid n/p_k^{e_k}$ holds.
Let $y=\vect{y}=F_{p_k^{e_k}}(x)$. Then, we have $p_k\mid d_{c_k}(n/p_k^{e_k})$ since $p_k\nmid c_k$ holds. That is, one has $d_{c_k}(n/p_k^{e_k}) \neq 1$. Thus, we have $y_{n/p_k^{e_k}}=0$ by (\ref{main4theo31}). On the other hand, using the definition of Frobenius operations we have 

\[ y_{n/p_k^{e_k}}=\sum_{[j,p_k^{e_k}]=n}\dfrac{jp_k^{e_k}}{n}x_j= p_k^{e_k}x_n\]
because if the integer $j$ satisfying $[j,p_k^{e_k}]=n$, then $j=n$ holds. Thus we can obtain $x_n=0$ since $R$ is $\mathbb{Z}$-torsion-free.

Hence, the element $x$ belongs to the image of $T_c$ by Proposition \ref{ImageTr}.
\end{proof}

\subsection{Multiplication and Frobenius operation}\label{SSMultiFrobenius}
In \S \ref{SSMultiFrobenius}, we discuss product groups and Frobenius operations to calculate $\lambda_t(\chi)(g)$ for a virtual character $\chi$ and an element $g$ of a finite group. Moreover, we discuss the multiplication of two elements of the image of truncation operations with Frobenius operations. First, we state the following lemma.

\begin{lemm}\label{valu1}
Let $\chi$ be a character of a finite group $G$ and let $g$ be an element of $G$. Then, for any integer $k\geq 1$, the $k$-th element of $z(\Lambda(\Eva{G}{g})(\lambda_t(\chi)))$ is $\chi(g^k)$.
\end{lemm}
\begin{proof}
By Proposition \ref{triplecomm}, we have 
\[ z(\Lambda(\Eva{G}{g})(\lambda_t(\chi)))=\Gh{{\Eva{G}{g}}}\circ z\circ\lambda_t(\chi
).\]
Hence, the $k$-th element of $z(\Lambda(\Eva{G}{g})(\lambda_t(\chi)))$ is $\Eva{G}{g}(\psi^k(\chi))$, which coincides with $\chi(g^k)$ by the definition of Adams operations on $CF(G)$.
\end{proof}

We state the following theorem.
\begin{theo}\label{main3MMM}
Let $G_1$ and $G_2$ be finite groups, and let $\chi_1$ {\rm(}resp. $\chi_2{\rm)}$ be a virtual character of $G_1$ {\rm(}resp. $G_2${\rm)}. Then, the virtual character $\chi_1\chi_2$ of $G_1\times G_2$ defined by $\chi_1\chi_2(g_1,g_2):=\chi_1(g_1)\chi_2(g_2)$ satisfies
\begin{eqnarray}\label{main3MM}
\lambda_t(\chi_1\chi_2)(g_1g_2)=\lambda_t(\chi_1)(g_1)\cdot_{\Lambda}\lambda_t(\chi_2)(g_2).
\end{eqnarray}
\end{theo}
\begin{proof}
We consider Adams operations. Let $g_1\in G_1$, $g_2\in G_2$, and 
\begin{eqnarray*}
a=\vect{a}&=&z(\Lambda(\Eva{G_1\times G_2}{(g_1,g_2)})(\lambda_t(\chi_1\chi_2))),\\ b=\vect{b}&=&z(\Lambda(\Eva{G_1}{g_1})(\chi_1)),\\
c=\vect{c}&=&z(\Lambda(\Eva{G_2}{g_2})(\chi_2)). 
\end{eqnarray*}
Then, for any integer $k\geq 1$, we have $a_k=\chi_1\chi_2((g_1,g_2)^k)=\chi_1(g_1^k)\chi_1(g_2^k)$, $b_k=\chi_1(g_1^k)$ and $c_k=\chi_2(g_2^k)$ by Lemma \ref{valu1}. Thus, we have $a_k=b_kc_k$, and hence, we have $a=bc$.
Since the complex field $\mathbb{C}$ is $\mathbb{Z}$-torsion free, the map $z:\Lambda(\mathbb{C})\rightarrow\Gh{\mathbb{C}}$ is an injective ring homomorphism. Hence, this theorem holds.
\end{proof}

Next, we prove the following theorem.
\begin{theo}\label{main3FF}
For any virtual character $\chi$ of a finite group $G$ and integer $k\geq 1$, we have $E_{Nr}(\lambda_t(\chi)(g^k))=F_k\circ E_{Nr}(\lambda_t(\chi)(g))$.
\end{theo}
\begin{proof}
For any integer $k\geq 1$, we have
\begin{eqnarray}\label{FM1}
\phi\circ E_{Nr}(\lambda_t(\chi)(g^k))&=&z(\Lambda(\Eva{G}{g^k})(\lambda_t(\chi)))
\end{eqnarray}
and
\begin{eqnarray}\label{FM2}
\phi\circ F_k\circ E_{Nr}(\lambda_t(\chi)(g))&=&F_k(z(\Lambda(\Eva{G}{g})(\lambda_t(\chi)))).
\end{eqnarray}
by Proposition \ref{phiVF}. Then $n$-th elements of both (\ref{FM1}) and (\ref{FM2}) coincide with $\chi (g^{nk})$ by Lemma \ref{valu1}. Hence, we have 
\[ \phi\circ E_{Nr}(\lambda_t(\chi)(g^k))=\phi\circ F_k\circ E_{Nr}(\lambda_t(\chi)(g)).\]

Since the map $\phi$ is injective, this proposition holds.
\end{proof}

Next, we discuss the problem to find another form of 
\begin{eqnarray}\label{Tab1}
T_{r}(x)\cdot_{Nr}T_{s}(y)
\end{eqnarray}
with Frobenius operation for any elements $x$ and $y$ of $\Nr{R}$. For example, we consider when $r=12$ and $s=18$. Put $x=\vect{x}, y=\vect{y}$ and $\alpha=\vect{\alpha}=x\cdot_{Nr} y$. We have
\begin{eqnarray*}
\alpha_{36}=\sum_{[i,j]=36}(i,j)x_iy_j=x_4y_9+3x_{12}y_9+2x_4y_{18}+6x_{12}y_{18}\\
\end{eqnarray*}
by the definition of multiplication of $\Nr{R}$. However, using Frobenius operation we have $\alpha_{36}=(x_4+3x_{12})(y_9+2y_{18})=u_{4}v_9$ where $u=\vect{u}=F_3(x)$ and $v=\vect{e}=F_2(y)$, which gives a factorization.
The calculation of (\ref{Tab1}) appears, for example, in the identities obtained from (\ref{main3MM}) using the map $E_{Nr}$ when virtual characters $\chi_1$ and $\chi_2$ are integer-valued characters. 
These results of this section will be used in the next section.

So, we calculate $T_{r}(x)\cdot_{Nr}T_{s}(y)$. However, we only need to calculate the $[r,s]$-th element since the following two lemmas hold.

\begin{lemm}\label{main5lemm1}
Let $x$ and $y$ be elements of $\Nr{R}$. For any integers $r,s\geq 1$, the element $T_r(x)\cdot_{Nr} T_s(y)$ belongs to the image of $T_{[r,s]}$.
\end{lemm}
\begin{proof}
By Proposition \ref{main4lemm00}, we have
\begin{eqnarray*}
T_{[r,s]}(T_r(x)\cdot_{Nr} T_s(y))&=&T_{[r,s]}\circ T_r(x)\cdot_{Nr}T_{[r,s]}\circ T_s(y)\\
&=&T_r(x)\cdot_{Nr} T_s(y).
\end{eqnarray*}
Hence, the element $T_r(x)\cdot_{Nr} T_s(y)$ belongs to the image of $T_{[r,s]}$.
\end{proof}

\begin{lemm}\label{main5lemm2}
For any integers $r,s,u\geq 1$ with $u\mid[r,s]$, we have $[(r,u), (s,u)]=u$.
\end{lemm}
\begin{proof}
We denote the prime decomposition of $r,s,u$ and $[(r,u),(s,u)]$ by $r=p_1^{r_1}\cdots p_q^{r_q}, s=p_1^{s_1}\cdots p_q^{s_q}, u=p_1^{u_1}\cdots p_q^{u_q}$ and $[(r,u),(s,u)]=p_1^{e_1}\cdots p_q^{e_q}$. Then, we have $u_i\leq$ max$(r_i, s_i)$ since $u\mid [r,s]$ holds.

By the definition of $e_i$, we have
\[ e_i=\mbox{max}(\mbox{min}(r_i,u_i), \mbox{min}(s_i,u_i)). \]
If $r_i\leq u_i$ and $s_i>u_i$, then one has $e_i=\mbox{max}(r_i,u_i)=u_i$. If $r_i>u_i$ and $s_i\leq u_i$. Thus one has $e_i=\mbox{max}(u_i,s_i)=u_i$. If $r_i>u_i$ and $ s_i>u_i$, then we have $e_i=u_i$. 

In any case, we have $e_i=u_i$ for any $i=1,\ldots,q$, that is, $[(r,u), (s,u)]=u$ holds.
\end{proof}
By Lemma \ref{main5lemm1} and Lemma \ref{main5lemm2}, the $n$-th element of $T_{r}(x)\cdot_{Nr} T_{s}(y)$ is equal to the $r$-th element of $T_{(a,r)}(x)\cdot_{Nr} T_{(b,r)}(y)$ for any divisor $n$ of $[r,s]$. Hence, we calculate the $[r,s]$-element of $T_{r}(x)\cdot_{Nr}T_{s}(y)$.

In the remainder of this section, we denote the $n$-th element of $x\in\Nr{R}$ by $x_n$ and the $n$-th element for $a\in\Gh{R}$ by $a_n$ for simplicity.
\begin{theo}\label{PMMaintheorem1}
Let $r,s \geq 1$ be integers, and we denote the prime decomposition of $r$ and $s$ by $r=p_1^{r_1}\cdots p_q^{r_q}$ and $s=p_1^{s_1}\cdots p_q^{s_q}$. We define integers $a_1,a_2,a_3,b_1,b_2,b_3\geq 1$ by
\begin{eqnarray*}
&{}&a_1=\prod_{i, r_i>s_i}p_i^{r_i},\quad a_2=\prod_{i, r_i=s_i}p_i^{r_i},\quad a_3=\prod_{i, r_i<s_i}p_i^{r_i},\\
&{}&b_1=\prod_{i, r_i>s_i}p_i^{s_i},\quad b_2=\prod_{i, r_i=s_i}p_i^{s_i},\quad b_3=\prod_{i, r_i<s_i}p_i^{s_i}.
\end{eqnarray*}
For integers $a_1,a_2,a_3,b_1,b_2$ and $b_3$, the followings hold,
\begin{eqnarray*}
&{}&r=a_1a_2a_3,\quad s=b_1b_2b_3,\\
&{}&(a_1,a_2)=(a_2,a_3)=(a_3,a_1)=1,\\
&{}&(b_1,b_2)=(b_2,b_3)=(b_3,b_1)=1,\\
&{}&b_1\mid a_1,\quad \Dpr(a_1)=\Dpr(a_1/b_1),\\ 
&{}&a_2=b_2,\\
&{}&a_3\mid b_3,\quad \Dpr(b_3)=\Dpr(b_3/a_3)
\end{eqnarray*}
where $\Dpr(n)$ is the set of all prime divisors of $n$ for $n\geq 1$. 

Suppose that $R$ is a $\mathbb{Q}$-algebra. Let $x$ and $y$ be elements of $\Nr{R}$. 
Then, we have
\[ (T_r(x)\cdot_{Nr} T_s(y))_{[r,s]}=\dfrac{1}{a_2}\sum_{d_2\mid a_2}\mu\Big(\dfrac{a_2}{d_2}\Big)(F_{a_1d_2}(x))_{a_3} (F_{d_2b_3}(y))_{b_1}.\]
In particular, we have
\[ (T_r(x)\cdot_{Nr} T_s(y))_{[r,s]}=(F_{a_1}(x))_{a_3} (F_{b_3}(y))_{b_1} \]
when $a_2=b_2=1$.
\end{theo}

\begin{proof}
We use the fact that the map $\phi:\Nr{R}\rightarrow \Gh{R}$ is bijective and preserves truncated operations $T_r$ and Frobenius operations $F_r$. We modify the left side. Then, we have 
\begin{eqnarray*}
&&(T_{a_1a_2a_3}(x)\cdot_{Nr}T_{b_1b_2b_3}(y))_{[a_1a_2a_3,b_1b_2b_3]}\\
&=&(\phi^{-1}(T_{a_1a_2a_3}\circ\phi(x)) (T_{b_1b_2b_3}\circ\phi(y)))_{a_1a_2b_3}\\ 
&=&\dfrac{1}{a_1a_2b_3}\sum_{d\mid a_1a_2b_3}\mu\Big(\dfrac{a_1a_2b_3}{d}\Big)(T_{a_1a_2a_3}\circ\phi(x))_{d} (T_{b_1b_2b_3}\circ\phi(y))_{d}\\
&=&\dfrac{1}{a_1a_2b_3}\sum_{d_1\mid a_1}\sum_{d_2\mid a_2}\sum_{d_3\mid b_3}\mu\Big(\dfrac{a_1}{d_1}\Big)\mu\Big(\dfrac{a_2}{d_2}\Big)\mu\Big(\dfrac{b_3}{d_3}\Big) (\phi(x))_{d_1d_2a_3} (\phi(y))_{b_1d_2d_3}\\
&=&\dfrac{1}{a_2}\sum_{d_2\mid a_2}\mu\Big(\dfrac{a_2}{d_2}\Big)\Big\{\dfrac{1}{a_1}\sum_{d_1\mid a_1}\mu\Big(\dfrac{a_1}{d_1}\Big)(F_{d_2a_3}\circ\phi(x))_{d_1}\Big\}\\
&& \Big\{\dfrac{1}{b_3}\sum_{d_3\mid b_3}\mu\Big(\dfrac{b_3}{d_3}\Big)(F_{b_1d_2}\circ\phi(x))_{d_3}\Big\}\\
&=&\dfrac{1}{a_2}\sum_{d_2\mid a_2}\mu\Big(\dfrac{a_2}{d_2}\Big)\Big\{\dfrac{1}{a_1}\sum_{d_1\mid a_1}\mu\Big(\dfrac{a_1}{d_1}\Big)(\phi\circ F_{b_1d_2}(x))_{d_1}\Big\}\\
&& \Big\{\dfrac{1}{b_3}\sum_{d_3\mid b_3}\mu\Big(\dfrac{b_3}{d_3}\Big)(\phi\circ F_{d_2a_3}(x))_{d_3}\Big\}\\
&=&\dfrac{1}{a_2}\sum_{d_2 \mid a_2}\mu\Big(\dfrac{a_2}{d_2}\Big)(F_{b_1d_2}(x))_{a_1}(F_{d_2a_3}(y))_{b_3}
\end{eqnarray*}
where the second equality follows from Propositions \ref{prophi} and \ref{lemm13}, the third and seventh equality follow from the definition of the map $\phi$ and the \Mobius\ inversion formula, the fourth and fifth equality follows from the definition of truncated operations and Frobenius operations, respectively, and the sixth equality follows from Proposition \ref{phiVF}.
\end{proof}

\section{A representation of a symmetric group with a multiplicative anti-symmetric matrix}

In this section, we calculate $E_{Nr}(\lambda_t(\chi))$ where a character $\chi$ is defined by the following as an example. We fix an integer $k$, and let $\masm=(q_{i,j})_{i,j}$ be a square matrix whose size is $k$, and $\masm$ satisfies $q_{i,j}q_{j,i}=1$ for any integers $i,j=1,\ldots,k$. This matrix $\masm$ was introduced in \cite[Appendix I.10]{BG}, and called a multiplicative anti-symmetric matrix.

In \S \ref{QQopeW}, we will define maps with divided Verchiebung operation and Frobenius operations, written by $W_r$, and state properties. In \S\ref{QQDef}, we define a representation of a symmetric group $S_n$ associated with $\masm$, and will be written by $\rho_{\masm,n}$. We will calculate the character $\chi_{\masm,n}$ of $\rho_{\masm,n}$ in \S\ref{QQmatrix} and $E_{Nr}(\lambda_t(\chi_{\masm,n}))((1\ 2\cdots\ n))$ in \S\ref{QQnecklace}. and we will discuss the calculation method of $E_{Nr}(\lambda_t(\chi_{\masm,n}))(\sigma)$ for all $\sigma\in S_n$ with Theorem \ref{PMMaintheorem1}.

\subsection{The map $W_r$}\label{QQopeW}
In \S \ref{QQopeW}, we suppose that $R$ is a $\mathbb{Q}$-algebra. First, we introduce new operations on $\Nr{R}$ and $\Gh{R}$ for \S\ref{QQnecklace}. 

\begin{defi}
For any integer $r\geq 1$, we define two maps
\begin{eqnarray*}
W_r&:&\Nr{R}\rightarrow \Nr{R},\\
W_r&:&\Gh{R}\rightarrow \Gh{R}
\end{eqnarray*}
by $W_r:=V'_r\circ F_r$.
\end{defi}
By this definition, the following holds.
\begin{prop}\label{phiW}
The map $\phi$ preserves the map $W_r$ for any integer $r\geq 1$.
\end{prop}
\begin{proof}
By Proposition \ref{phiVF}, we have
\[ \phi\circ W_r=\phi\circ \dfrac{1}{r}V'_r\circ F_r=\dfrac{1}{r}V'_r\circ \phi\circ F_r=\dfrac{1}{r}V'_r\circ F_r\circ\phi=W_r\circ\phi.\]
Thus, the map $\phi$ preserves $W_r$.
\end{proof}

By Propositions \ref{proVF} and \ref{proTV2}, we have the following proposition.
\begin{prop}\label{proW}
The followings hold for any integers $r,s \geq 1$.
\begin{itemize}
\item [{\rm (1)}]$F_r \circ W_s=W_{s/(r,s)}\circ F_r$,
\item [{\rm (2)}]$T_r\circ W_s=0\ {\rm if}\ s\nmid r$.
\end{itemize}
\end{prop}
\begin{proof}
(1) By Proposition \ref{proVF}, we have $F_r\circ W_s=F_r\circ V'_s\circ F_s=V'_{s/(r,s)}\circ F_{r/(r,s)}\circ F_{s}=W_{s/(r,s)}\circ F_r.$

(2) By Proposition \ref{proTV2}, we have $T_r\circ W_s=T_r\circ V'_s\circ F_s=0$.
\end{proof}

We will use the following lemma in \S \ref{QQnecklace}.
\begin{lemm}\label{lemmW}
For any element $a=\vect{a}\in\Gh{R}$ and integer $r\geq 1$, we have
\[ b_n=\begin{cases} a_n & \mbox{if}\ r\mid n, \\ 0 & \mbox{if}\ r\nmid n, \end{cases}\]
where $\vect{b}=W_r(a)$.
\end{lemm}
\begin{proof}
Put $c=\vect{c}=F_r(a)$. If an integer $n\geq 1$ satisfies $r\mid n$ then $b_n=c_{n/r}=a_n$ holds. If $r\nmid n$, then $b_n=0$ holds since $b=V_r'(c)$ holds.
\end{proof}

\subsection{Definition}\label{QQDef}
In \S \ref{QQDef}, we define a representation $\rho_{\masm,n}$. First, we consider a set of generators of $S_n$.
\begin{lemm}[\rm{\cite[p.88]{Sag}}]\label{generatorS}
The followings hold for a symmetric group $S_n$.
\begin{itemize}
\item [{\rm (1)}] The group $S_n$ is generated by $n-1$ elements $\sigma_1,\ldots,\sigma_{n-1}$ where $\sigma_i=(i\ i+1)$ for any integer $i=1,\ldots,n-1$.
\item [{\rm (2)}] All $\sigma_i$'s satisfy the following relations.
\begin{eqnarray*}
\sigma_i^2=1, && i=1,\ldots,n-1, \\
(\sigma_i \sigma_{i+1})^3=1, && i=1,\ldots,n-2, \\
(\sigma_i\sigma_j)^2=1, && |i-j|\geq 2.
\end{eqnarray*}
\item [{\rm (3)}] Suppose that $n\geq 2$. The group $S_n$ is isomorphic to a group $G_n$ which is generated by $n-1$ elements $\tau_1,\ldots,\tau_{n-1}$ satisfying $\tau_i^2=1$ for any integer $i=1,\ldots,n-1$, $(\tau_i\tau_{i+1})^3=1$ for any integer $i=1,\ldots,n-2$ and $(\tau_i\tau_j)^2=1$ for any integer $i,j=1,\ldots,n-1$ with $|i-j|\geq 2$.
\end{itemize}
\end{lemm}
\begin{proof}
First, we prove (1). All elements of $S_n$ can be written as some product of cycles. For $(i_1\ i_2 \cdots i_q)$, we have
\[ (i_1\ i_2 \cdots i_q)=(i_1\ i_2)(i_2\ i_3)\cdots (i_{q-1}\ i_q).\]
For an element $(i\ j)$ where $i,j=1,\ldots, n$ with $i<j$, we have $(i\ j)=\sigma\sigma_j\sigma^{-1}$ where $\sigma=(i\ i+1\cdots j)=\sigma_i\cdots \sigma_{j-1}$. Then, all elements of $S_n$ are generated by $\sigma_1,\ldots,\sigma_{n-1}$.

It is obvious that the statement (2) holds. So, we prove (3). By the definition of $G_n$ and (2), there exists a surjective group homomorphism $\iota:G_n\rightarrow S_n$ such that $\iota(\tau_i)=\sigma_i$ holds for any integer $i=1,\ldots, n$. 

We prove $|G_n|=n!$ by induction on $n\geq 2$. If $n=2$, we have $G_2=\{1,\tau_1\}$. In particular, the cardinality of $G_2$ is $2$, that is, the statement (3) holds when $n=2$. 

Next, let $n\geq 3$ be an integer and we show $|G_n|=n!$ with the assumption that $G_{n-1}$, which is the subgroup generated by $\tau_1,\ldots,\tau_{n-2}$, satisfies $|G_{n-1}|=(n-1)!$. 

Put 
\[ H_i=\begin{cases} G_{n-1}\tau_{n-1}\tau_{n-2}\cdots\tau_{i} &\mbox{if}\ i=1,\ldots,n-1,\\ G_{n-1} & \mbox{if}\ i=n, \\ \end{cases} \]
 and $H=\bigcup_{i=1}^{n-1}H_i$. Now, we consider $H_i\tau_j$. If $i=n$, then $H_1\tau_j\in H$. We assume $i<n$. If $j>i+1$ then $H_i\tau_j=H_i\subset H$ holds. If $j=i+1$, then $H_i\tau_j=H_{i+1}\subset H$. If $j=i$ then $H_i\tau_j=H_{i-1}$. If $j<i$, then
\begin{eqnarray*}
H_i\tau_j&=&(G_{n-1}\tau_{n-1}\tau_{n-2}\cdots\tau_{i})\tau_j\\
&=&G_{n-1}\tau_{n-1}\tau_{n-2}\cdots\tau_{j+1}\tau_j\tau_{j-1}\tau_j\tau_{j-2}\cdots\tau_i\\
&=&G_{n-1}\tau_{n-1}\tau_{n-2}\cdots\tau_{j+1}\tau_{j-1}\tau_{j}\tau_{j-1}\tau_{j-2}\cdots\tau_i\\
&=&G_{n-1}\tau_{n-1}\tau_{n-2}\cdots\tau_{j+1}\tau_{j}\tau_{j-1}\tau_{j-2}\cdots\tau_i\subset H.
\end{eqnarray*}
In any case, we have $H_i\tau_j\in H$. Thus, all generators $\tau_1,\ldots,\tau_{n-1}$ belong to $H$. Then, we have $G_n=H$. For each $H_i$ and $H_j$, the set $H_i\cap H_j$ has no elements if $i\neq j$. Hence, we have $|G_n|=|H|=n|G_{n-1}|=n!$.
\end{proof}

In \S 4, we denote by $V$ a vector space over $\mathbb{C}$ with a basis $\{v_1,\ldots,v_k\}$.

\begin{defi}
We define a bijective linear map $F:V\otimes V\rightarrow V\otimes V$ by
\[ F(v_i\otimes v_j):=q_{i,j}(v_j\otimes v_i) \]
for any integers $i,j=1,\ldots,k$.
\end{defi}

\begin{prop}\label{Young}
The map $F$ satisfies $F\circ F=\identitymap{V}$ and the Yang Baxter equation. That is, 
\[ (\identitymap{V}\otimes F)\circ (F\otimes \identitymap{V})\circ (\identitymap{V}\otimes F)= (F\otimes \identitymap{V})\circ (\identitymap{V}\otimes F)\circ (F\otimes \identitymap{V}) \]
holds as a linear map on $V\otimes V\otimes V$.
\end{prop}
\begin{proof}
First, we show $F\circ F=\identitymap{V}$. For any integers $i,j=1,\ldots,n$, we have $F\circ F(v_i\otimes v_j)=q_{i,j}F(v_j\otimes v_i)=q_{i,j}q_{j,i}(v_i\otimes v_j)=v_i\otimes v_j$ by the definition of $\masm$. Then, we have $F\circ F=\identitymap{V}$.

Next, we show that the map $F$ satisfies the Yang Baxter equation. For any integers $i_1,i_2,i_3=1,\ldots,k$, we have
\begin{eqnarray*}
&{}&(\identitymap{V}\otimes F)\circ (F\otimes \identitymap{V})\circ (\identitymap{V}\otimes F)(v_{i_1}\otimes v_{i_2}\otimes v_{i_3})\\
&=&q_{i_2,i_3}(\identitymap{V}\otimes F)\circ (F\otimes \identitymap{V})(v_{i_1}\otimes v_{i_3}\otimes v_{i_2}) \\
&=&q_{i_1,i_3}q_{i_2,i_3}(\identitymap{V}\otimes F)(v_{i_3}\otimes v_{i_1}\otimes v_{i_2}) \\
&=&q_{i_1,i_2}q_{i_1,i_3}q_{i_2,i_3}(v_{i_3}\otimes v_{i_2} \otimes v_{i_1}),
\end{eqnarray*}
and
\begin{eqnarray*}
&{}&(F\otimes \identitymap{V})\circ (\identitymap{V}\otimes F)\circ (F\otimes \identitymap{V})(v_{i_1}\otimes v_{i_2}\otimes v_{i_3})\\
&=&q_{i_1,i_2}(F\otimes \identitymap{V})\circ (\identitymap{V}\otimes F)(v_{i_2}\otimes v_{i_1}\otimes v_{i_3}) \\
&=&q_{i_1,i_3}q_{i_1,i_2}(\identitymap{V}\otimes F)(v_{i_2}\otimes v_{i_3}\otimes v_{i_1}) \\
&=&q_{i_2,i_3}q_{i_1,i_3}q_{i_1,i_2}(v_{i_3}\otimes v_{i_2} \otimes v_{i_1}).
\end{eqnarray*}
Hence, we have $(\identitymap{V}\otimes F)\circ (F\otimes \identitymap{V})\circ (\identitymap{V}\otimes F)= (F\otimes \identitymap{V})\circ (\identitymap{V}\otimes F)\circ (F\otimes \identitymap{V})$.
\end{proof}

\begin{defi}
For each integer $n\geq 1$, we define a representation $\rho_{\masm,n}:S_n\rightarrow GL(V^{\otimes n})$ by
\[ \rho_{\masm, n}(\sigma_i):=\identitymap{V}^{\otimes (i-1)}\otimes F \otimes \identitymap{V}^{\otimes(n-i-1)} \]
where $\sigma_i=(i\ i+1)$.
\end{defi}
By Lemma \ref{generatorS} and Proposition \ref{Young}, this representation is well-defined.

\subsection{A representation matrix of $\rho_{\masm,n}(\sigma)$}\label{QQmatrix}
In \S \ref{QQmatrix}, for any $\sigma\in S_n$, we consider a representation matrix of the linear map $\rho_{\masm,n}(\sigma)$ with respect to the basis $B_n$ on $V^{\otimes n}$, which is defined by
\begin{eqnarray*}
&\{&v_{1}\otimes\cdots\otimes v_{1},\ldots,v_{1}\otimes\cdots\otimes v_{k},\\
&{}&\ldots,\\
&{}&v_{k}\otimes\cdots\otimes v_{1},\ldots,v_{k}\otimes\cdots\otimes v_{k}\}.
\end{eqnarray*}

We define the following two serieses of matrices.

\begin{defi}
For any integers $i,j=1,\ldots,k$ and $l\geq 0$, we define matrices $F_{l,j,i}\in M(k^l,\mathbb{C})$ by induction on $l\geq 0$, and
\begin{eqnarray*}
F_{0,j,i}&:=&\delta_{j,i},\\
F_{l,j,i}&:=&q_{j,i}\left( \begin{array}{@{\,}ccc@{\,}}
 & \bigzerou & \\
F_{l-1,1,i} & \ldots & F_{l-1,k,i} \\
 & \bigzerol & 
\end{array}\right)\ (l\geq 1)
\end{eqnarray*}
where matrices $F_{l-1,1,i},\ldots,F_{l-1,k,i}$ lies on $j$-th row of $F_{l,j,i}$ for $l\geq 1$.
\end{defi}
\begin{defi}
For any integer $l\geq 1$, we define a matrix $P(\masm,l) \in M(k^l,\mathbb{C})$ by $P(\masm,l):=(F_{l-1,j,i})_{i,j}$.
\end{defi}


First, we consider $\rho_{\masm,n}(\sigma)$ when $\sigma$ is a cycle of length $p$.

\begin{prop}\label{repmat}
Let $m$ and $p$ be integers with $m+(p-1)<n$. Then the matrix of the linear map $\rho_{\masm,n}((m\ m+1\cdots m+(p-1)))$ with respect to the basis $B_n$ is equal to $I_k^{\otimes(m-1)}\otimes P(\masm,p)\otimes I_k^{(n-m+(p-1))}$.
\end{prop}

To prove Proposition \ref{repmat}, we state the following lemmas.

\begin{lemm}\label{lemm1-1}
The matrix of the linear map $F:V\otimes V\rightarrow V\otimes V$ with respect to the basis $B_2$ is $P(\masm,2)$.
\end{lemm}
\begin{proof}
For any integer $j=1,\ldots, k$, we have
\begin{eqnarray*}
&{}&(F(v_j\otimes v_1),\ldots,F(v_j\otimes v_k))\\
&=&(q_{j,i}(v_1\otimes v_j),\ldots,q_{j,n}(v_k\otimes v_j))\\
&=&\sum_{p=1}^k (v_p\otimes v_1,\ldots, v_p\otimes v_k)F_{1,j,p} \\
&=&(v_1\otimes v_1,\ldots,v_1\otimes v_k,\ldots,v_k\otimes v_1,\ldots, v_k\otimes v_k)\left ( \begin{array}{@{\,}c@{\,}}F_{1,j,1} \\ \vdots \\ F_{1,j,k} \end{array} \right ).
\end{eqnarray*}
Hence, this lemma holds.
\end{proof}

\begin{lemm}\label{lemm1-11}
For any integer $l\geq 2$,
\begin{eqnarray}\label{eqn1}
F_{l,j,i}=(F_{l-1,j,i}\otimes I_k)(I_k^{\otimes(l-2)}\otimes P(\masm,2))
\end{eqnarray}
holds.
\end{lemm}
\begin{proof}
We show this lemma by induction on $l\geq 2$. If $l=2$, then the identity (\ref{eqn1}) holds by
\begin{eqnarray*}
(F_{1,j,i}\otimes I_k)P(\masm,2)&=&q_{j,i}\left( \begin{array}{@{\,}ccc@{\,}}
\bigzerou & & \bigzerou \\
& I_k& \\
\bigzerol & & \bigzerol 
\end{array}\right)
\left( \begin{array}{@{\,}ccc@{\,}}
F_{1,1,1}& \ldots & F_{1,k,1}\\
\vdots& \ddots & \vdots \\
F_{1,1,k}& \ldots & F_{1,k,k}
\end{array}\right) \\
&=&\left( \begin{array}{@{\,}ccc@{\,}}
 & \bigzerou & \\
F_{1,1,i} & \ldots & F_{1,k,i} \\
 & \bigzerol & 
\end{array}\right)=F_{2,j,i}
\end{eqnarray*}
where the matrix $I_k$ is the $(j,i)$ block in $(F_{1,j,i}\otimes I_k)$ and matrices $F_{1,1,i},\ldots, F_{1,k,i} $ lies on $j$-th row.

We assume that the identity (\ref{eqn1}) holds for any integer $m\geq l-1$. By this assumption, we have
\begin{eqnarray*}
&{}&(F_{l-1,j,i}\otimes I_k)(I_k^{\otimes(l-2)}\otimes P(\masm,2)) \\
&=&\left( \begin{array}{@{\,}ccc@{\,}}
 & \bigzerou & \\
q_{j,i}(F_{l-2,1,i}\otimes I_k) & \ldots & q_{j,i}(F_{l-2,k,i}\otimes I_k) \\
 & \bigzerol & 
\end{array}\right)\\
&{}&\left( \begin{array}{@{\,}ccc@{\,}}
I_k^{\otimes(l-3)}\otimes P(\masm,2)& &\bigzerou \\
& \ddots & \\
\bigzerol & & I_k^{\otimes(l-3)}\otimes P(\masm,2)
\end{array}\right) \\
&=&q_{j,i}\left( \begin{array}{@{\,}ccc@{\,}}
 & \bigzerou & \\
F_{l-1,1,i} & \ldots & F_{l-1,k,i} \\
 & \bigzerol & 
\end{array}\right)=F_{l,j,i}
\end{eqnarray*}
where matrices $q_{j,i}(F_{l-2,1,i}\otimes I_k),\ldots,q_{j,i}(F_{l-2,k,i}\otimes I_k)$ and $F_{l-1,1,i}, \ldots,$ $F_{l-1,k,i}$ lie on $j$-th row. Hence, the identity (\ref{eqn1}) holds for any $l\geq 2$. 
\end{proof}

\begin{lemm}\label{lemm2-1}
For any integer $l\geq 2$, 
\[ P(\masm,l)=(P(\masm,l-1)\otimes I_k)(I_k^{\otimes(l-2)}\otimes P(\masm,2))\]
holds.
\end{lemm}
\begin{proof}
We use (\ref{eqn1}) to prove this lemma. If $l=2$, then we have
\[ (P(\masm,1)\otimes I_k)P(\masm,2)=I_{k^2}P(\masm,2)=P(\masm,2). \]
If $l\geq 3$, then we have
\begin{eqnarray*}
&{}&(P(\masm,l-1)\otimes I_k)(I_k^{\otimes(l-2)}\otimes P(\masm,2)) \\
&=&(F_{l-2,j,i}\otimes I_k)_{i,j} \left( \begin{array}{@{\,}ccc@{\,}}
I_k^{\otimes(l-3)}\otimes P(\masm,2)& & \bigzerou\\
& \ddots & \\
\bigzerol& & I_k^{\otimes(l-3)}\otimes P(\masm,2)
\end{array}\right) \\
&=&((F_{l-2,j,i}\otimes I_k)(I_k^{\otimes(l-3)}\otimes P(\masm,2)))_{i,j}=(F_{l-1,j,i})_{i,j}=P(\masm,l).
\end{eqnarray*}
Hence, this lemma holds.
\end{proof}

\begin{proof}[Proof of Proposition \ref{repmat}]
Let $\rho'_{\masm,n}(\sigma)$ be a matrix of the linear map $\rho_{\masm,n}(\sigma)$ with respect to the basis $B_n$. We prove Proposition \ref{repmat} by induction on $p\geq 2$. 

We assume $p=2$. In this case, we have $\rho'_{\masm,n}((m\ m+1))=\identitymap{V}^{\otimes i-1}\otimes P(\masm,2) \otimes \identitymap{V}^{\otimes(n-i+1)}$ by Lemma \ref{lemm1-1}. Thus, this proposition holds if $p=2$.

For any integer $p\geq 2$, we suppose that this proposition holds when $p-1$. For any integer $m\geq 1$, 
\[ (m\ m+1\cdots m+(p-1))=(m\ m+1\cdots m+(p-2))(m+(p-2)\ m+(p-1))\]
holds. Thus,
\begin{eqnarray*}
&{}&\rho'_{\masm,n}((m\ m+1\cdots m+(p-1)))\\
&=&\rho'_{\masm,n}((m\ m+1\cdots m+(p-2)))(\rho'_{\masm,n}((m+(p-2)\ m+(p-1)))\\
&=&(I_k^{\otimes(m-1)}\otimes (P(\masm,p-1)\otimes I_k)\otimes I_k^{\otimes(n-m+(p-1))}) \\
&{}&(I_k^{\otimes(m-1)}\otimes (I^{\otimes (p-2)}\otimes P(\masm,2))\otimes I_k^{\otimes(n-m+(p-1))})\\
&=& I_k^{\otimes(m-1)}\otimes P(\masm,p)\otimes I_k^{\otimes (n-m+(p-1))}
\end{eqnarray*}
holds. Hence, this proposition holds.
\end{proof}

By Proposition \ref{repmat}, the following corollary holds.

\begin{cor}\label{char1}
Put $\pi_1=(1\ 2\cdots \lambda_1), \pi_2=(\lambda_1+1\ \cdots\ \lambda_1+\lambda_2),\ldots,\pi_q=(\lambda_1+\cdots+\lambda_{q-1}+1\ \cdots\ n)$ $\in S_n$. Then, a matrix of $\rho_{\masm,n}(\pi_1\cdots\pi_q)$ with respect to the basis $B_n$ is $P(\masm,\lambda_1)\otimes\cdots\otimes P(\masm,\lambda_q)$.
\end{cor}

\subsection{Calculation of $E_{Nr}(\lambda_t(\chi))$}\label{QQnecklace}
In \S \ref{QQnecklace}, we show the calculation method of $E_{Nr}(\lambda_t(\chi_{\masm,n}))$, where $\chi_{\masm,n}$ is the character of the representation $\rho_{\masm,n}$.

One has $\phi (E_{Nr}(\lambda_t(\chi_{\masm,n})))=z\circ\lambda_t(\chi_{\masm,n})$, hence its $n$-th element is $\psi^n(\chi_{\masm,n})$. Then, we investigate $\chi_{\masm,n}$ in order to calculate $E_{Nr}(\lambda_t(\chi))$. We investigate $\chi_{\masm,n}$ by the following propositions.

\begin{prop}\label{char2}
For any integer $l\geq 2$, we have
\[ \mathrm{Tr}(P(\masm,l))=\begin{cases} \mathrm{Tr}(\masm) & \mbox{if}\ 2\mid l, \\ k & \mbox{if}\ 2\nmid l.\end{cases}\]
In particular, the number $\mathrm{Tr}(P(\masm,l))$ is an integer.
\end{prop}
\begin{proof}
First, we prove $\mathrm{Tr}(F_{l,j,i})=\delta_{j,i}q_{j,i}^l$ by induction on $l \geq 0$. If $l=0$, then we have $\mathrm{Tr}(F_{0,j,i})=\delta_{j,i}=\delta_{j,i}q_{j,i}^{0}$. Let $l\geq 1$ be an integer and we assume that $\mathrm{Tr}(F_{l-1,j,i})=\delta_{j,i}q_{j,i}^{l-1}$. Then, we have
\[ \mathrm{Tr}(F_{l,j,i})=\mathrm{Tr}(q_{j,i}F_{l-1,j,i})=q_{j,i}\delta_{j,i}q_{j,i}^{l-1}=\delta_{j,i}q_{j,i}^l.\]
For $P(\masm,l)$, we have
\[ \mathrm{Tr}(P(\masm,l))=\sum_{i=1}^k \mathrm{Tr}(F_{l-1,i,i})=\sum_{i=1}^k q_{i,i}^{l-1} \]
Thus, for any integer $i=1,\ldots, k$ we have
\[ q_{i,i}^{l-1}=\begin{cases}q_{i,i} &\mbox{if }2\mid l, \\ 1 & \mbox{if}\ 2\nmid l.\end{cases}\]
Hence, we have
\[ \mathrm{Tr}(P(\masm,l))=\sum_{i=1}^k q_{i,i}^{l-1}=\begin{cases} \mathrm{Tr}(\masm) & \mbox{if}\ 2\mid l, \\ k & \mbox{if}\ 2\nmid l.\end{cases}\]
By the definition of $\masm$, we have $q_{i,i}q_{i,i}=1$, which means that $q_{i,i}=1$ or $q_{i,i}=-1$ holds for any integer $i=1,\ldots,k$. Thus, the number $\mathrm{Tr}(\masm)$ is an integer. Hence, the number $\mathrm{Tr}(P(\masm,l))$ is an integer.
\end{proof}
By Corollary \ref{char1}, Proposition \ref{char2} and Lemma \ref{KroLem}, we have the following proposition. 
\begin{prop}\label{char3}
Let $\sigma$ be an element of $S_n$, and let $\{\lambda_1,\ldots,\lambda_m\}$ be the cycle structure of $\sigma$. Then, we have 
\[ \chi_{\masm,n}(\sigma)=k^{s_1(\sigma)}\mathrm{Tr}(\masm)^{s_2(\sigma)}\]
where $s_1(\sigma)$ is the number of $l=1,\ldots,m$ such that $2\nmid \lambda_l$ holds, and $s_2(\sigma)$ is the number of $l=1,\ldots,m$ such that $2\mid \lambda_l$ holds. In particular, the character $\chi_{\masm,n}$ is an integer-valued character.
\end{prop}

\begin{rmk}
Remark that if matrices $\masm=(q_{i,j})_{i,j}, \masm'=(q'_{i,j})_{i,j}\in M(k,\mathbb{C})$ satisfying $q_{i,j}q_{j,i}=1$ and $q'_{i,j}q'_{j,i}=1$ for any integers $i,j=1,\ldots k$, if $\mathrm{Tr}(\masm)=\mathrm{Tr}(\masm')$ then two representation $\rho_{\masm,n}$ and $\rho_{\masm',n}$ are isomorphic.
\end{rmk}

First, we calculate $E_{Nr}(\lambda_t(\chi_{\masm,n})(\sigma))$ when $\sigma=(1\ 2\cdots n)$. Next, we calculate when $\sigma=(1\ 2\cdots n)^r$ for any integer $r\geq 1$. Finally, we consider all $\sigma\in S_n$ by using Theorem \ref{PMMaintheorem1}.

To calculate $E_{Nr}(\lambda_t(\chi))$, we use the following identity.
\begin{defi}
For any integer $l$, we define $M(l)=\vect{m}\in\Nr{\mathbb{C}}$ by
\[ m_n:=\dfrac{1}{n}\sum_{d\mid n}\mu\Big(\dfrac{n}{d}\Big)l^n\]
for any integer $n\geq 1$.
\end{defi}
Note that if $l\geq 1$, then $m_n$ is said to be the necklace polynomial in \cite{MR}. 

For $M(l)$, the following proposition holds.

\begin{lemm}\label{QQpoly}
The followings hold.
\begin{itemize}
\item [{\rm (1)}] The $n$-th element of $\phi(M(l))$ is $l^n$ for any integer $l$ and $n\geq 1$.
\item [{\rm (2)}] We have $F_r(M(l))=M(l^r)$ for any integer $l$ and $r\geq 1$.
\end{itemize}
\end{lemm}
\begin{proof}
For the statement (1), the $n$-th element of $\phi(M(l))$ is 
\[ \sum_{d\mid n}d m_d=\sum_{d\mid n}\sum_{d'\mid d}\mu\Big(\dfrac{d}{d'}\Big)l^{d'}=l^n\]
by the \Mobius\ inversion formula, where $\vect{m}=M(l)$. The statement (2) was proved by Metropolis and Rota \cite[p.104, Proposition 2]{MR}.
\end{proof}
First, we consider $E_{Nr}(\lambda_t(\chi_{\masm,n})(\sigma))$ when $\sigma=(1\ 2\cdots\ n)$.

\begin{prop}\label{QQ1}
Let $n\geq 1$ be an integer and put $\sigma=(1\ 2\cdots n)$. Then,
\[ E_{Nr}(\lambda_t(\chi_{\masm,n})(\sigma))=T_n(M(\mathrm{Tr}(\masm))+W_{2^c}(M(k)-M(\mathrm{Tr}(\masm))))\]
holds where $c\geq 1$ is the maximum integer satisfying $2^c\mid n$.
\end{prop}
\begin{proof}
We show that two elements
\[a=\vect{a}=\phi(E_{Nr}(\lambda_t(\chi_{\masm,n})(\sigma)))\]
and 
\[b=\vect{b}=\phi\circ T_n(M(\mathrm{Tr}(\masm))+W_{2^c}(M(k)-M(\mathrm{Tr}(\masm))))\]
are equal. To show it, we prove three identities $a_d=a_{(n,d)}$ for any integer $d\geq 1$, $b_d=b_{(n,d)}$ for any integer $d\geq 1$ and $a_d=b_d$ if $d\mid n$.

First, we consider the element $a$. By Theorem \ref{YauTheo}, we have $\phi(E_{Nr}(\lambda_t(\chi_{\masm,n})(\sigma)))=z(\lambda_t(\chi_{\masm,n})(\sigma))$. Thus, 
\[ a_d=\psi^d(\chi_{\masm,n})(\sigma)=\chi_{\masm,n}(\sigma^d).\]
By Proposition \ref{char3}, $\chi_{\masm,n}$ is an integer-valued character. Thus we have $a_d=a_{(n,d)}$ by Lemma \ref{lemm12}. If $d\mid n$, then $\sigma^d$ is a $d$-times multiple of cycles of length $n/d$ by Lemma \ref{SN1}. Hence, we have
\begin{eqnarray}\label{QQ11}
\chi_{\masm,n}(\sigma^d)= \begin{cases} k^{d} &\mbox{if}\ 2 \nmid n/d,\\ (\mathrm{Tr}(\masm))^{d} & \mbox{if}\ 2\mid n/d.\end{cases}
\end{eqnarray}

Next, we consider the element $b$. By Propositions \ref{lemm13} and \ref{phiW},
\[\vect{b}=T_n(\phi(M(\mathrm{Tr}(\masm)))+W_{2^c}(\phi(M(k))-\phi(M(\mathrm{Tr}(\masm)))))\]
holds. Thus, the element $b$ belongs to the image of $T_n$, which means that $b_d=b_{(n,d)}$ holds by Proposition \ref{ImageTr} (2). If $d\mid n$, then
\begin{eqnarray}\label{QQ12}
b_d=\begin{cases} k^d & \mbox{if}\ 2^c\mid d, \\ (\mathrm{Tr}(\masm))^d & \mbox{if}\ 2^c\nmid d\end{cases}
\end{eqnarray}
by Lemmas \ref{lemmW} and \ref{QQpoly}.

A divisor $d$ of $n$ satisfies $2\nmid n/d$ if and only if $2^c\mid d$ holds. Thus, $a_d=b_d$ holds for any divisor $d$ of $n$ by (\ref{QQ11}) and (\ref{QQ12}). Hence, we have $a=b$. Since $\mathbb{C}$ is $\mathbb{Z}$-torsion free, this proposition holds.
\end{proof}

Next, we have the following proposition.
\begin{prop}\label{repmat2}
For any divisor $r$ of $n$ we have 
\[ F_r(\alpha)=T_{n/r}( M(\mathrm{Tr}(\masm)^r)+W_{2^{c-f}}(M(k^r)-M(\mathrm{Tr}(\masm)^r)))\]
where $f$ is the maximum number satisfying $2^f\mid r$ and 
\[ \alpha=T_n(M(\mathrm{Tr}(\masm))+W_{2^c}(M(k)-M(\mathrm{Tr}(\masm)))).\]
\end{prop}
\begin{proof}
One has
\begin{eqnarray*}
F_r(\alpha)&=&F_r\circ T_n(M(\mathrm{Tr}(\masm))+W_{2^c}(M(k)-M(\mathrm{Tr}(\masm))))\\
&=&T_{n/(n,r)}\circ F_r(M(\mathrm{Tr}(\masm))+W_{2^c}(M(k)-M(\mathrm{Tr}(\masm))))\\
&=&T_{n/r}( F_r(M(\mathrm{Tr}(\masm)))+F_r\circ W_{2^c}(M(k)-M(\mathrm{Tr}(\masm))))\\
&=&T_{n/r}( F_r(M(\mathrm{Tr}(\masm)))+W_{2^{c-f}}(F_r(M(k))-F_r(M(\mathrm{Tr}(\masm)))))\\
&=&T_{n/r}( M(\mathrm{Tr}(\masm)^r)+W_{2^{c-f}}(M(k^r)-M(\mathrm{Tr}(\masm)^r))).
\end{eqnarray*}
where the second, fourth or fifth equality follows from Proposition \ref{ComCF}, Proposition \ref{proW} or Lemma \ref{QQpoly} (2), respectively.
\end{proof}
By Proposition \ref{repmat2}, we have the following corollary.

\begin{cor}\label{repmat3}
For any integers $n\geq 1$ and $r\geq 1$,
\[ E_{Nr}(\lambda_t(\chi_{\masm,n})(\sigma^r))=T_{n/r}( M(\mathrm{Tr}(\masm)^r)+W_{2^{a-b}}(M(k^r)-M(\mathrm{Tr}(\masm)^r)))\]
holds where $\sigma=(1\ 2\cdots n)$, and integer $a$ or $b$ is the maximum number satisfying $2^a\mid n$ holds or $2^b\mid n/r$ holds, respectively.
\end{cor}
\begin{proof}
Put $\alpha=T_n(M(\mathrm{Tr}(\masm))+W_{2^a}(M(k)-M(\mathrm{Tr}(\masm))))$. Then
\[ F_r(\alpha)=T_{n/r}( M(\mathrm{Tr}(\masm)^r)+W_{2^{a-b}}(M(k^r)-M(\mathrm{Tr}(\masm)^r)))\]
by Proposition \ref{repmat2}. On the other hand, we have $F_r(\alpha)=E_{Nr}(\lambda_t(\chi_{\masm,n})(\sigma^r))$. Hence, this corollary holds.
\end{proof}

We calculated $\lambda_t(\chi_{\masm,n}))(\sigma)$ where $\sigma=(1\ 2\cdots n)$ in Proposition \ref{QQ1} and where the cycle structure $\{\lambda_1,\dots,\lambda_m\}$ of an element $\sigma\in S_n$ satisfies $\lambda_i=\lambda_j$ for any integer $i,j=1,\ldots m$ in Corollary \ref{repmat3}.

Finally, we state the calculation method when $\sigma$ is the product of two disjoint cycles by the following proposition. Hence, we can calculate $E_{Nr}(\lambda_t(\chi_{\masm,n}))(\sigma)$ for any $\sigma\in S_n$, because all elements of $S_n$ can be written as a product of cycles.

\begin{prop}\label{Jonesprod}
For any $\sigma\in S_m$ and $\tau\in S_n$, we have 
\[ \lambda_t(\chi_{m+n,\masm})(\sigma\tau)=\lambda_t(\chi_{m,\masm})(\sigma)\cdot_{\Lambda}\lambda_t(\chi_{n,\masm})(\tau). \]
where we identify the subgroup $\{\sigma\in S_{m+n}\mid \sigma(i)=i,\ i=m+1,\dots,n\}$ or $\{\tau\in S_{m+n}\mid \tau(i)=i,\ i=1,\dots,m\}$ of $S_{m+n}$ with $S_m$ or $S_n$, respectively.
\end{prop}
\begin{proof}
First, we show $\Res{S_{m+n}}{S_m\times S_n}(\chi_{m+n,\masm})=\chi_{n,\masm}\chi_{m,\masm}$ as a character of the product group $S_m\times S_n$. For any $\sigma'\in S_m$ and let $\tau'\in S_n$, we have $s_1(\sigma'\tau')=s_1(\sigma')+s_1(\tau')$ and $s_2(\sigma'\tau')=s_1(\sigma')+s_2(\tau')$. Thus, we have $\chi_{m+n,\masm}(\sigma'\tau')=\chi_{n,\masm}(\sigma)\chi_{m,\masm}(\tau)$. Hence, we have $\Res{S_{m+n}}{S_m\times S_n}(\chi_{m+n,\masm})=\chi_{n,\masm}\chi_{m,\masm}$.

We use this fact. By Lemma \ref{valu1}, the $d$-th element of $z(\lambda_t(\chi_{m+n,\masm})(\sigma\tau))$ is 
\begin{eqnarray*}
\psi^d(\chi_{m+n,\masm})(\sigma\tau)&=&\chi_{m+n,\masm}(\sigma^d \tau^d)\\
&=&\Res{S_{m+n}}{S_m\times S_n}(\chi_{m+n,\masm})((\sigma^d,\tau^d))\\
&=&\chi_{m,\masm}(\sigma^d)\chi_{n,\masm}(\tau^d),
\end{eqnarray*}
which is the product of $d$-th elements of $z(\lambda_t(\chi_{m,\masm})(\sigma))$ and $z(\lambda_t(\chi_{n,\masm})(\tau))$. Thus, we have $ z(\lambda_t(\chi_{m+n,\masm})(\sigma\tau))=z(\lambda_t(\chi_{m,\masm})(\sigma)\cdot_{\Lambda}\lambda_t(\chi_{n,\masm})(\tau))$.

Since $\mathbb{C}$ is $\mathbb{Z}$-torsion free, the map $z$ is injective. It follows that we have $\lambda_t(\chi_{m+n,\masm})(\sigma\tau)=\lambda_t(\chi_{m,\masm})(\sigma)\cdot_{\Lambda}\lambda_t(\chi_{n,\masm})(\tau)$.
\end{proof}


By Proposition \ref{Jonesprod}, we can calculate $\lambda_t(\chi_{\masm, n})(\sigma)$ for all $\sigma\in S_n$, and we can apply Theorem \ref{PMMaintheorem1}.
\newpage
\def\thesection{\Alph{section}}

\setcounter{section}{0}
\section{Appendix}

In this section, we state some notions and propositions which are used in  this paper.

\subsection{Vector spaces}

For any $A\in M(k,\mathbb{C})$, we denote the trace of $A$ by $\mathrm{Tr}(A)$. 

For any $A,B\in M(k,\mathbb{C})$, we denote the Kronecker product of $A$ and $B$ by $A\otimes B$. 

For any complex vector space $V$, we denote the set of all bijective linear maps on $V$ by $GL(V)$.

Let $F:V\rightarrow V$ be a linear map of a vector space $V$. We define $\mathrm{Tr}(F)$ by the trace of the matrix of $F$ with respect to a basis $\{x_1,\ldots,x_n\}$. Note that $\mathrm{Tr}(F)$ does not depend on the basis $\{x_1,\ldots,x_n\}$.

For any vector spaces $V$ and $W$, we denote the direct sum of $V$ and $W$ by $V\oplus W$, and the tensor product of $V$ and $W$ by $V\otimes W$. 

Let $V$ and $W$ be vector spaces and let $F:V\rightarrow W$ be a linear map. For any integer $n\geq 1$, we denote $n$-times tensor product of $V$ by $V^{\otimes n}$, and we define a linear map $F^{\otimes n}:V^{\otimes n}\rightarrow W^{\otimes}$ by $F^{\otimes n}(v_1\otimes\cdots\otimes v_n):=F(v_1)\otimes\cdots\otimes F(v_n)$ for any $v_1,\ldots,v_n\in V$.

Let $V_1,V_2,W_1$ and $W_2$ be vector spaces and let $F_1:V_1\rightarrow W_1$ and $F_2:V_2\rightarrow W_2$ be linear maps. Then, we define a linear map $F_1\otimes F_2:V_1\otimes V_2\rightarrow W_1\otimes W_2$ by $F_1\otimes F_2(v_1\otimes v_2):=F_1(v_1)\otimes F_2(v_2)$ for any $v_1\in V_1$ and $v_2\in V_2$.

For the Kronecker product of two matrices, the following lemma holds from Lemma (v) in \cite[p.83]{Knu}.
\begin{lemm}\label{KroLem}
For any $A\in M(k_1,\mathbb{C})$ and $B\in M(k_2,\mathbb{C})$, we have $\mathrm{Tr}(A\otimes B)=\mathrm{Tr}(A)\mathrm{Tr}(B)$.
\end{lemm}

\subsection{Finite groups}
Let $G$ be a finite group.

\begin{lemm}\label{APP1}
Two elements $g_1,g_2\in G$ generate the same cyclic group $\circg{g_1}=\circg{g_2}$ if and only if there exists an integer $k\geq 1$ such that $g_1^k=g_2$ and $(k, O(g_1))=1$ hold. 
\end{lemm}

\begin{lemm}\label{APP2}
Let $g_1$ and $g_2$ be elements of $G$ with $\circg{g_1}\cap \circg{g_2}=\{1\}$ and $g_1g_2=g_2g_1$. Then, we have $O(g_1g_2)=[O(g_1),O(g_2)]$.
\end{lemm}
\begin{proof}
First, we show $(g_1g_2)^{[O(g_1),O(g_2)]}=g_1^{[O(g_1),O(g_2)]}g_2^{[O(g_1),O(g_2)]}=1$. In addition, if an arbitrary integer $k$ satisfies $(g_1g_2)^k=1$, then $O(g_1)\mid k$ and $O(g_2)\mid k$ hold by $\circg{g_1}\cap\circg{g_2}=\{1\}$. Hence, we have $[O(g_1),O(g_2)]\mid k$. This means $O(g_1g_2) \mid [O(g_1), O(g_2)]$.

Conversely, $(g_1g_2)^{O(g_1g_2)}=1$ implies that $g_1^{O(g_1g_2)} = g_2^{-O(g_1g_2)}\in\circg{g_1}\cap\circg{g_2}=\{1\}$ which means $O(g_1) \mid O(g_1g_2)$ and $O(g_2) \mid O(g_1 g_2)$. Hence we have $[O(g_1), O(g_2)] \mid O(g_1g_2)$. Hence, $O(g_1g_2) = [O(g_1),O(g_2)]$ holds.
\end{proof}

\subsection{For symmetric groups}
In this paper, we denote the symmetric group on $n$-letters by $S_n$. For more detail, see \cite[p.124]{Knu}.

An element $\sigma$ of $S_n$ is a cycle of length $q$, there exists a subset $\{i_1,\ldots,i_q\}\subset\{1,\ldots, n\}$ such that $\sigma(i_1)=i_2, \sigma(i_2)=i_3,\ldots,\sigma(i_q)=i_1$ and $\sigma(m)=m$ hold where $m$ is not the form $i_j$. In the case, we write $\sigma=(i_1\ i_2\cdots i_n)$.

An element $\sigma\in S_n$ can be written as a product of disjoint cycles: $\sigma=(i_1\ i_2\cdots\ i_q)(j_1\ j_2\cdots\ j_r)\cdots(k_1\ k_2\cdots\ k_s)$ where all integers which have the form $i_p, j_q,\ldots,k_s$ are distinct. The cycle structure of $\sigma$ is a partition $(\lambda_1,\ldots,\lambda_m)$ of $n$ where $\lambda_i$ is the size of the cycles in the decomposition.

By Lemma \ref{APP2}, we have the following lemma.
\begin{lemm}\label{APP3}
An element of $\sigma\in S_n$ which has the cycle structure $(\lambda_1,\ldots,\lambda_m)$ satisfies that $O(\sigma)$ is the least common multiple of $\lambda_1,\ldots,\lambda_m$
\end{lemm}

For the cycle structure of $\sigma\in S_n$, it is known the following lemma.

\begin{lemm}{\rm \cite[p.125]{Knu}}\label{APP4}
Two elements of $S_n$ are conjugate if and only if they have the same cycle structure.
\end{lemm}

\subsection{For representations of finite groups}
Let $G$ be a finite group.

We define a group homomorphism $\rho:G\rightarrow GL(V)$ by a representation of $G$, where the vector space $V$ is a complex vector space whose dimension is finite. We define the degree of $\rho$ by $\dim_{\mathbb{C}}(V)$. 

The character of $\rho$, which is denoted by $\chi$, is the map $\chi:G\rightarrow\mathbb{C}$ defined by $\chi(g)=\mathrm{Tr}(\rho(g))$ for any $g\in G$.

Next, we define a $G$-set. A $G$-set $X$ is a set equipped with a map 
\[ \iota:G\times X\rightarrow X,\quad \iota(g,x)=gx\]
which satisfies
\[ g_1(g_2x)=(g_1g_2)x,\quad 1x=x\]
for any $g_1,g_2\in G$ and $x\in X$ (Note that we have mentioned it in \S 6.1).

Put $X=\{x_1,\ldots,x_n\}$. We define the permutation representation associated with $X$ by the representation $\rho_X:G\rightarrow GL(V)$ where $V$ is the $n$-dimensional vector space which is formally generated by $\{x_1,\ldots,x_n\}$, and $\rho_X(g)(x_i):=gx_i$ holds for any $g\in G$. Finally, we define the permutation character associated with $X$ by the character of $\rho_X$.

For any $g\in G$, the matrix of $\rho_X(g)$ with respect to the basis $\{x_1,\ldots,x_n\}$ is written by $(x_{i,j})_{1\leq i,j\leq n}$, satisfies 
\[ x_{i,j}=\begin{cases}1 & \mbox{if}\ x_j=gx_i, \\ 0 &\mbox{otherwise.}\end{cases}\]
Let $\chi$ be the permutation character associated with $X$. Then, $\chi(g)$ is the number of $x\in X$ such that $gx=x$ holds. 

\subsection{For commutative rings}

We consider the following two properties on a commutative ring $R$.

\begin{defi}
A commutative ring $R$ is $\mathbb{Z}$-torsion free if the map $r\mapsto nr$ is injective for any integer $n\geq 1$.
A commutative ring $R$ is a $\mathbb{Q}$-algebra if $R$ contains a subring which is isomorphic to $\mathbb{Q}$ as rings.
\end{defi}

With the \Mobius\ function, the following proposition holds.

\begin{prop}[\Mobius\ inversion formula]
Let $\vect{a}$ and\\ $\vect{b}$ be infinite vectors of a commutative ring $R$. Two elements $a$ and $b$ satisfy 
\[ b_n=\sum_{d\mid n}a_d\]
for any integer $n\geq 1$ if and only if 
\[ a_n=\sum_{d\mid n}\mu\Big(\dfrac{n}{d}\Big)b_d\]
holds for any integer $n \geq 1$.
\end{prop}

We state the Gauss Lemma.

\begin{lemm}[\text {\cite [Lemma 3.17]{Ste}}]\label{Gauss}
Let $f$ be a polynomial over $\mathbb{Z}$ that is irreducible over $\mathbb{Z}$. Then $f$, considered as a polynomial over $\mathbb{Q}$, is also irreducible over $\mathbb{Q}$.
\end{lemm}

\subsection{For semirings and their ring completion}
In this section, we tell the notion of the ring completion. For more detail, see \cite[\S 9.3]{Hus}.

We define a semiring by a triple $(S,\alpha,\mu)$ where $S$ is a set, the map $\alpha:S\times S\rightarrow S$ is the addition function usually denoted by $\alpha(a,b)=a+b$, and the map $\mu:S\times S\rightarrow S$ is the multiplication function usually denoted by $\mu(a,b)=ab$, and satisfies all axioms of a ring except the existence of negative or additive inverse. For simplicity, we denote a semiring $(S,\alpha,\mu)$ by $S$.

We define a semiring homomorphism from a semiring $S$ to a semiring $S'$ by a map $f:S\rightarrow S'$ such that $f(a+b)=f(a)+f(b)$, $f(ab)=f(a)f(b)$ and $f(0)=0$ hold.

\begin{prop}
For any semiring $S$, there exists a pair $(S^*,\theta)$, where $S^*$ is a ring and $\theta:S\rightarrow S^*$ is a semiring homomorphism such that if for any ring $R$ and a semiring homomorphism $f:S\rightarrow R$, there exists a ring homomorphism $g:S^*\rightarrow R$ such that $g\circ\theta=f$ holds.
\end{prop}


Tomoyuki Tamura

Graduate School of Mathematics, Kyushu-University, 

Nishi-ku Fukuoka,  819-0395, Japan.

E-mail: t-tamura@math.kyushu-u.ac.jp

\end{document}